\documentclass[12pt,reqno]{amsart}
\usepackage{amscd, amssymb, pb-diagram, amsmath}
\usepackage{amsthm}

\usepackage{tikz}
\usepackage[margin=2.0cm]{geometry}
\usepackage{color}
\usepackage{tensor}
\usetikzlibrary{arrows,matrix,backgrounds,decorations.markings}
\numberwithin{equation}{section}

\def\ker{\operatorname{Ker}}
\def\im{\operatorname{Im}}

\def\dim{\operatorname{dim}}

\def\id{\operatorname{id}}

\def\min{\operatorname{min}}

\def\id{\operatorname{id}}

\def\R{\mathbb{R}}
\def\N{\mathbb{N}}
\def\Z{\mathbb{Z}}

\def\Q{\mathbb{Q}}

\def\AA{\mathcal{A}}

\def\QQ{\mathcal{Q}}

\newcommand{\ep}{\varepsilon}
\newcommand{\fib}[4]{#1 \tensor[_{#2}]{\times}{_{#3}} #4}

\newcommand{\Corr}[7]{
\begin{scope}[xshift=#1cm,yshift=#2cm]
\node at (0,0) {$#3$};
\draw[->] (1,1) -- node[pos=0.4,left] {$#4$\,} (0.25,0.3);
\node at (1.5,1.3) {$#5$};
\draw[->] (2,1) -- node[pos=0.4,right] {\,$#6$} (2.75,0.3);
\node at (3,0) {$#7$};
\end{scope}}

\newcommand{\Correq}[8]{
\begin{scope}[xshift=#1cm,yshift=#2cm]
\node at (0,0) {$#3$};
\node at (3,0) {$#4$};
\draw[->] (0,1.5) -- node[pos=0.4,left] {$#5$} (0,0.3);
\draw[->] (0.3,1.5) -- node[pos=0.4,left] {$#6$\,} (2.75,0.3);
\draw[->] (2.75,1.5) -- node[pos=0.4,right] {\,\,\,$#7$} (0.25,0.3);
\draw[->] (3,1.5) -- node[pos=0.4,right] {$#8$} (3,0.3);
\end{scope}}

\tikzstyle{vertex}=[circle]
\tikzstyle{goto}=[->,shorten >=1pt,>=stealth,semithick]

\newtheorem{thm}{Theorem}[section]
\newtheorem{cor}[thm]{Corollary}

\newtheorem{lemma}[thm]{Lemma}
\newtheorem{prop}[thm]{Proposition}

\theoremstyle{definition}
\newtheorem{definition}[thm]{Definition}
\newtheorem{deflem}[thm]{Definition/Lemma}

\theoremstyle{remark}
\newtheorem{remark}[thm]{Remark}

\newtheorem{example}[thm]{Example}

\newtheorem*{Acknowledgements}{Acknowledgements}

\begin{document}

\begin{abstract}
We initiate the study of correspondences for Smale spaces. Correspondences are shown to provide a notion of a generalized morphism between Smale spaces and are a special case of finite equivalences. Furthermore, for shifts of finite type, a correspondence is related to a matrix which intertwines the adjacency matrices of the shifts. This observation allows us to define an equivalence relation on all Smale spaces which reduces to shift equivalence for shifts of finite type. Several other notions of equivalence are introduced on both correspondences and Smale spaces; a hierarchy between these equivalences is established. Finally, we provide several methods for constructing correspondences and provide specific examples.
\end{abstract}

\title[Dynamical Correspondences for Smale spaces]{Dynamical Correspondences for Smale spaces}

\author{Robin J. Deeley}
\address{Robin J. Deeley,  Laboratorie de Mathematiques, Universite Blaise Pascal, Clermont-Ferrand II, France}
\email{robin.deeley@gmail.com}
\author[D. Brady Killough]{D. Brady Killough}
\address{Brady Killough, Department of Mathematics and Physics, Mount Royal University, Calgary, Alberta, Canada T3E 6K6}
\email{bkillough@mtroyal.ca}
\author[Michael F. Whittaker]{Michael F. Whittaker}
\address{Michael F. Whittaker, School of Mathematics and Statistics,
University of Glasgow, Glasgow, Scotland Q12 8QW} 
\email{Mike.Whittaker@glasgow.ac.uk}

\thanks{This research was supported by the ANR Project SingStar and the Australian Research Council (DP130100490).}


\maketitle

\section{Introduction}

Smale spaces were defined by Ruelle as a purely topological description of the basic sets of Smale's Axiom A diffeomorphisms on a compact manifold \cite{Rue1,Sma}. Smale spaces are topological dynamical systems with a local hyperbolic product structure. Shifts of finite type are the zero dimensional examples of Smale spaces and, because of Bowen's theorem \cite{Bow}, are the basic building blocks of the theory.

The main goal of this paper is the introduction of a class of generalized morphisms between Smale spaces. Given two Smale spaces $(X,\varphi)$ and $(Y,\psi)$, there is a correspondence from   $(X,\varphi)$ to $(Y,\psi)$ if there is a third Smale space $(M,\mu)$ and a diagram 
\begin{equation}\label{corr_intro}
\begin{tikzpicture}
\Corr{0}{0}{(X,\varphi)}{\pi_u}{(M,\mu)}{\pi_s}{(Y,\psi)}
\end{tikzpicture}
\end{equation}
where $\pi_u$ is a u-bijective map and $\pi_s$ is an s-bijective map; these types of maps will be discuss below, and in more detail in Section \ref{sec:mapsonSmale}. In this paper we show that correspondences allow us to generalize notions that are specific to the combinatorial nature of shifts of finite type to the language of Smale spaces; prototypical examples include shift equivalence and strong shift equivalence.

The recent development of Putnam's homology theory for Smale spaces \cite{put}, which generalizes Krieger's dimension groups for shifts of finite type, plays an important role in this paper. In particular, the functorial properties of Putnam's homology theory emphasize the importance of  s- and u-bijective maps, which are related to, but are less general than Fried's notion of s- and u-resolving maps \cite{Fri}. An s-bijective map $\pi_s:(X,\varphi) \to (Y,\psi)$ is a factor map that restricts to an isomorphism between specific stable sets of $X$ and $Y$; likewise a u-bijective map restricts to an isomorphism between specific unstable sets of $X$ and $Y$ \cite[Definition 2.5.5]{put}.  The functoriality of Putnam's homology theory with respect to s- and u-bijective maps implies that a correspondence between Smale spaces $(X,\varphi)$ and $(Y,\psi)$ induces natural maps between the homology groups $H_*(X,\varphi)$ and $H_*(Y,\psi)$. Putnam's homology theory and its functorial properties are reviewed in Section \ref{putHomRev}.

One of the prevailing themes of this paper is the generalization of notions for shifts of finite type to all Smale spaces. A finite equivalence between shifts of finite type is a diagram like \eqref{corr_intro} except that the maps $\pi_u$ and $\pi_s$ are only assumed to be finite-to-one factor maps; two Smale spaces are called finitely equivalent if there exists a finite equivalence between them. There is a rather strong analogy between finite equivalences and certain non-negative integer matrices going back to work of Furstenberg and Parry \cite[Theorem 1 and Lemma 2]{Par}. Given an irreducible graph $G$, we let $A_G$ denote the adjacency matrix of $G$ and we let $(\Sigma_G, \sigma)$ denote the associated shift of finite type. Our starting point for generalization in this direction is a slight reformulation of \cite[Theorem 8.3.8]{LM} (also see Theorem \ref{838LMinCor} below):

\begin{thm} \label{introThm}
Let $G$ and $H$ be irreducible graphs. Then the following are equivalent
\begin{enumerate}
\item $(\Sigma_G, \sigma)$ and $(\Sigma_H, \sigma)$ are finitely equivalent;
\item $(\Sigma_G, \sigma)$ and $(\Sigma_H, \sigma)$ have equal entropy;
\item There exists a nonnegative integer matrix $F\ne 0$ such that $F A_G  = A_H F$; and
\item There is a correspondence from $(\Sigma_G, \sigma)$ to $(\Sigma_H, \sigma)$.
\end{enumerate}
\end{thm}
In general, Theorem 2.5.3 in \cite{put} shows that s- and u-bijective maps on Smale spaces are finite-to-one, so a correspondence is a special case of a finite equivalence between Smale spaces. Theorem \ref{introThm} implies that the {\it existence} of a finite equivalence, a correspondence, or an intertwining matrix between irreducible shifts of finite type are equivalent. However, a finite equivalence between shifts of finite type does not canonically induce a map at the dimension group level. In contrast, an intertwining matrix or a correspondence between shifts of finite type does define a map at the dimension group level. This is the first indication that correspondences give a more natural generalization of an intertwining matrix than finite equivalences.

Returning to the case of Smale spaces, Putnam's homology theory is not (or at least is not known to be) functorial for arbitrary finite-to-one factor maps, but as noted above is functorial for s- and u-bijective maps and hence also for correspondences. Since definitions involving correspondences can be applied to all Smale spaces, it is natural to translate notions and results phrased in terms of matrices for shifts of finite type into the language of correspondences -- examples include shift equivalence and strong shift equivalence.

We say that two Smale spaces $(X,\varphi)$ and $(Y,\psi)$ are $H$-equivalent if there is a correspondence from $(X,\varphi)$ to $(Y,\psi)$ and a correspondence from  $(Y,\psi)$ to $(X,\varphi)$ so that the composition induces the identity map at the level of Putnam's homology theory in all possible orders (see Definition \ref{equivSmaleSpaceLevel}). In Theorem \ref{HequiImpliesShiftSFT} we show that $H$-equivalence of Smale spaces generalizes shift equivalence. In the same spirit, strong equivalence of Smale spaces (see Definition \ref{equivSmaleSpaceLevel}) is equivalent to strong shift equivalence as shown in Corollary \ref{SFTtostongshiftequiv}.


A further motivation for the definition of correspondences comes from index theory. Namely, from the notion of correspondences in KK-theory due to Connes and Skandalis \cite{CS} (also see \cite{EM}). However, an understanding of KK-theory is not required to understand the results of this paper. Classes in KK-theory can be viewed as generalized morphisms; correspondences of Smale spaces should also be viewed in this way.

All in all, the goal of this paper is to introduce correspondences for Smale spaces, discuss their basic properties, and explore the strengths and weaknesses of the analogy between intertwining nonnegative integer matrices and correspondences in the case of shifts of finite type. Furthermore, we compare existence results concerning correspondences and finite equivalences. In the case of the latter, Bowen's theorem implies that results in the shift of finite type case generalize with little or no change to the Smale space case. The situation for the existence of a correspondence is more subtle as it involves the stable and unstable sets, compare Theorems \ref{finiteEqualEntropy} and \ref{generalized LM 8.3.8}.

The content of the paper is as follows. Section \ref{prelims} briefly introduces the reader to Smale spaces and their homology groups. In Section \ref{Sec:Finite_equiv} we review finite equivalence and almost conjugacy for Smale spaces. Section \ref{sec:correspondences} introduces the main new concept in the paper, correspondences for Smale spaces. These provide a notion of a generalized morphism between two Smale spaces, and we show that in the shift of finite type case, existence of a correspondence is equivalent to existence of a finite equivalence. In Section \ref{sec:equivalences}, we introduce several equivalence relations between correspondences and use these to define various notions of invertibility for correspondences. In turn these notions of invertibility lead to notions of equivalence for Smale spaces. Section \ref{sec:constructioncorr} provides several general methods for constructing correspondences as well as a special case of a K\"{u}nneth formula for Putnam's homology theory for Smale spaces. We use these constructions to show that our new notion of H-equivalence between Smale spaces is strictly stronger than having isomorphic homology theory, even with the automorphisms associated to the dynamics accounted for. In Section \ref{sec:implications} we study the implications of such equivalences for Smale spaces, and relate them back to the implications for shifts of finite type by providing a diagram similar to the diagram on p.261 of \cite{LM}, which shows the implications in the shift of finite type case. Finally, we look ahead to some unanswered questions in Section \ref{sec:outlook}, especially to the notion of an equivalence of correspondences related to the $K$-theory of the $C^*$-algebras associated with Smale spaces. We note that our notion of a correspondence is not related to the notion of C*-correspondence.

\begin{Acknowledgements}
We thank Ian Putnam for many interesting and useful discussions about the content of this paper, and for his guidance and support during this early part of our academic careers. In particular, we thank Ian for remarks concerning the definitions of equivalence considered here, and for suggesting the K\"{u}nneth formula in order to construct examples. 
\end{Acknowledgements}

\section{Preliminaries}\label{prelims}

In this section we briefly provide background on Smale spaces and Putnam's homology for Smale spaces. These results all appear in Putnam's \emph{A Homology Theory for Smale Spaces} \cite{put} in much more detail. The reader would be well advised to have a copy on hand while reading this paper. We use \cite{LM} and \cite[Section 2.2]{put} for notation regarding shifts of finite type; when they do not agree we follow \cite{put}.

\subsection{Smale spaces}\label{Sec:Smale}

A Smale space $(X,\varphi)$ is a dynamical system consisting of a homeomorphism $\varphi$ on a compact metric space $X$ such that the space is locally the product of a coordinate that contracts under the action of $\varphi$ and a coordinate that expands under the action of $\varphi$. The precise definition is as follows.

\begin{definition}[{\cite[p.19]{put}, \cite{Rue1}}]\label{defSmaSpa}
A Smale space $(X,\varphi)$ consists of a compact metric space $X$ with metric $d$ along with a homeomorphism $\varphi: X \to X$ such that there exist constants $ \ep_{X} > 0, 0<\lambda < 1$ and a continuous bracket map
\[ \{(x,y) \in X \times X \mid d(x,y) \leq \ep_{X}\} \mapsto [x, y] \in X \]
satisfying the bracket axioms:
\begin{itemize}
\item[B1] $\left[ x, x \right] = x$,
\item[B2] $\left[ x, [ y, z] \right] = [ x, z]$,
\item[B3] $\left[ [ x, y], z \right] = [ x,z ]$, and
\item[B4] $\varphi[x, y] = [ \varphi(x), \varphi(y)]$;
\end{itemize}
for any $x, y, z$ in $X$ when both sides are defined. In addition, $(X,\varphi)$ is required to satisfy the contraction axioms:
\begin{itemize}
\item[C1] For $x,y \in X$ such that $[x,y]=y$, we have $d(\varphi(x),\varphi(y)) \leq \lambda d(x,y)$ and
\item[C2] For $x,y \in X$ such that $[x,y]=x$, we have $d(\varphi^{-1}(x),\varphi^{-1}(y)) \leq \lambda d(x,y)$.
\end{itemize}
\end{definition}

The local product structure isn't entirely obvious from the definition. Suppose $x \in X$ and $ 0 < \ep \leq \ep_{X}$, then we define 
\begin{align*}
 X^{s}(x, \ep) & := \{ y \in X \mid d(x,y) < \ep, [y,x] =x \}, \\ 
X^{u}(x, \ep) & := \{ y \in X \mid d(x,y) < \ep, [x,y] =x \}.
\end{align*}
The set $X^{s}(x, \ep)$ is called the local stable set and the set $X^{u}(x, \ep)$ is called the local unstable set. For $x,y \in X$ such that $d(x,y)<\ep_X/2$, the bracket map $[x,y]$ is the unique point where the local stable set of $x$ intersects the local unstable set of $y$ and vice versa, as in Figure \ref{Bracket intersection}.
\begin{figure}[ht]
\begin{center}
\begin{tikzpicture}
\tikzstyle{axes}=[]
\begin{scope}[style=axes]
	\draw[<->] (-3,-1) node[left] {$X^s(x,\ep_X)$} -- (1,-1);
	\draw[<->] (-1,-3) -- (-1,1) node[above] {$X^u(x,\ep_X)$};
	\node at (-1.2,-1.4) {$x$};
	\node at (1.1,-1.4) {$[x,y]$};
	\pgfpathcircle{\pgfpoint{-1cm}{-1cm}} {2pt};
	\pgfpathcircle{\pgfpoint{0.5cm}{-1cm}} {2pt};
	\pgfusepath{fill}
\end{scope}
\begin{scope}[style=axes]
	\draw[<->] (-1.5,0.5) -- (2.5,0.5) node[right] {$X^s(y,\ep_X)$};
	\draw[<->] (0.5,-1.5) -- (0.5,2.5) node[above] {$X^u(y,\ep_X)$};
	\node at (0.7,0.2) {$y$};
	\node at (-1.6,0.2) {$[y,x]$};
	\pgfpathcircle{\pgfpoint{0.5cm}{0.5cm}} {2pt};
	\pgfpathcircle{\pgfpoint{-1cm}{0.5cm}} {2pt};
	\pgfusepath{fill}
\end{scope}
\end{tikzpicture}
\caption{The local coordinates of $x,y \in X$ and their bracket maps}
\label{Bracket intersection}
\end{center}
\end{figure}
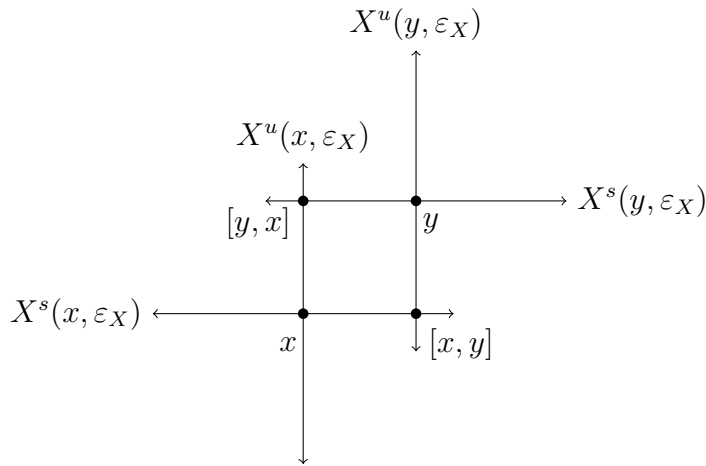

It is worth noting that if a bracket map exists on $(X,\varphi)$ then it is unique. In this paper we study Smale spaces with topological recurrence conditions -- specifically, {\em non-wandering}, {\em irreducible}, and {\em mixing} Smale spaces, see \cite[Definitions 2.1.3, 2.1.4, and 2.1.5]{put}.

For each $x \in X$ the local stable and unstable sets extend to global stable and unstable sets $X^s(x)$ and $X^u(x)$ as follows
\begin{align*}
X^s(x) & := \{y \in X | \lim_{n \rightarrow +\infty} d(\varphi^n(x),\varphi^n(y)) = 0\}, \\
X^u(x) & := \{y \in X | \lim_{n \rightarrow +\infty} d(\varphi^{-n}(x),\varphi^{-n}(y)) = 0\}.
\end{align*}
If $x$ and $y$ are in the same stable or unstable equivalence class we denote this by $x \sim_s y$ and $x \sim_u y$, respectively. We note that for any $x \in X$ we have $X^s(x,\ep) \subset X^s(x)$ and similarly for the unstable sets. 
Moreover, for $y \in X$ we often consider $X^s(y)$ as a topological space itself. The open sets $\{X^s(x,\ep) \mid x \in X^s(y), 0 <\ep< \ep_X\}$ form a neighbourhood base for a locally compact and Hausdorff topology on $X^s(y)$.

The prototypical examples of a Smale spaces are the shifts of finite type, which we now introduce in a manner that will be convenient for defining Putnam's homology theory. A greatly expanded version of our treatment can be found in \cite[Section 2.2]{put}.

A graph $G=(G^0,G^1,i,t)$ consists of finite sets $G^0$ and $G^1$, called the vertices and edges, such that each edge $e \in G^1$ is given by a directed edge from $i(e) \in G^0$ to $t(e) \in G^0$, see \cite[Definition 2.2.1]{put}. A vertex $v$ in a graph is called a \emph{source} if $t^{-1}\{v\}=\varnothing$ and a \emph{sink} if $i^{-1}\{v\}=\varnothing$. Throughout this paper $G$ will always denote a finite directed graph with no sources or sinks. Given a graph $G$, let $A_G$ denote the adjacency matrix of $G$; that is, $(A_G)_{vw}=|\{e\in G^1 \mid i(e)=v \text{ and } t(e)=w\}|$.

The standard definition of a shift of finite type is given in \cite[Definition 2.1.1]{LM}. However, an equivalent, and more convenient for our purposes, definition is to suppose $G$ is finite directed graph with no sources or sinks. Then a shift of finite type $(\Sigma_G,\sigma)$ is defined as the space of bi-infinite sequences of paths 
\[
\Sigma_G:=\{(e^k)_{k \in \Z} \mid e^k \in G \text{ and } t(e^i)=i(e^{i+1})\},
\]
with the shift map $\sigma(e)^k=e^{k+1}$, which is a homeomorphism from $\Sigma_G$ back to itself. 
A shift of finite type becomes a Smale space with the bracket map and metric defined in \cite[Definition 2.2.5]{put}. It is routine to verify that two points in $\Sigma_G$ are stably equivalent if they are right tail equivalent and unstably equivalent if they are left tail equivalent.
Finally, Theorem 2.2.8 in \cite{put} shows that every totally disconnected Smale space is conjugate to a shift of finite type.

\subsection{Maps on Smale spaces}\label{sec:mapsonSmale}

Let $(X,\varphi)$ and $(Y,\psi)$ be Smale spaces. A \emph{map} $\pi:(Y,\psi) \to (X,\varphi)$ consists of a continuous function $\pi:Y \to X$ such that $\pi \circ \psi=\varphi \circ \pi$. A map is called a \emph{factor map} if $\pi:Y \to X$ is surjective.
In \cite[Theorem 2.3.2]{put}, Putnam proves that maps are compatible with the bracket map, and then a straightforward argument shows that for any $y \in Y$ we have $\pi(Y^s(y)) \subset X^s(\pi(y))$.
A map is said to be \emph{s-resolving} if for any $y \in Y$ the restriction $\pi|_{Y^s(y)}:Y^s(y) \to X^s(\pi(y))$ is injective \cite{Fri}.
An s-resolving map is called \emph{s-bijective} if for all $y \in Y$ the restriction $\pi|_{Y^s(y)}:Y^s(y) \to X^s(\pi(y))$ is bijective. Of course, there are analogous definitions of u-resolving and u-bijective. In \cite[Theorem 2.5.6]{put}, Putnam shows that if $\pi:(Y,\psi) \to (X,\varphi)$ is an s-bijective map, then $(\pi(Y),\varphi|_{\pi(Y)})$ is a Smale space. Moreover, suppose  $(Y,\psi)$ is a non-wandering Smale space and $\pi:(Y,\psi) \to (X,\varphi)$ is an s-resolving map, then $\pi$ is s-bijective (\cite[Theorem 2.5.8]{put}).

\begin{definition} \label{defOfFibPro}
Let $(X,\varphi)$, $(Y,\psi)$, and $(Z,\zeta)$ be Smale spaces and $\pi_1: (X,\varphi) \rightarrow (Z,\zeta)$ and $\pi_2: (Y,\psi) \rightarrow (Z,\zeta)$ be finite-to-one factor maps. Then the fibre product (of $(X,\varphi)$ and $(Y,\psi)$ over the maps $\pi_1$ and $\pi_2$) is defined to be the space 
\[ \fib{X}{\pi_{1}}{\pi_{2}}{Y}:=\{ (x,y)\in X\times Y \: | \: \pi_1(x)=\pi_2(y)\}.
\]
It is a dynamical system using the map $\varphi \times \psi$. We denote this dynamical system as
\[
\fib{(X,\varphi)}{\pi_{1}}{\pi_{2}}{(Y,\psi)}
\]
There are projection maps $P_1: \fib{(X,\varphi)}{\pi_{1}}{\pi_{2}}{(Y,\psi)} \rightarrow (X,\varphi)$ and $P_2:\fib{(X,\varphi)}{\pi_{1}}{\pi_{2}}{(Y,\psi)} \rightarrow (Y,\psi)$.
\end{definition}

Fibre products are key to defining Putnam's homology theory. The first reason is that  \cite[Theorem 2.4.2]{put} implies that the fibre product of two Smale spaces is again a Smale space. The second is \cite[Theorem 2.5.13]{put}: suppose $\fib{(X,\varphi)}{\pi_{1}}{\pi_{2}}{(Y,\psi)}$ is the fibre product in Definition \ref{defOfFibPro} and $P_1, P_2$ are the projection maps, then if $\pi_1$ is s-bijective so is $P_2$. Similarly if $\pi_2$ is u-bijective, then so is $P_1$.


The fibre product can be iterated, and we will require $N$-fold fibre products over a single space. Suppose $(X,\varphi)$ and $(Y,\psi)$ are Smale spaces and $\pi:(Y,\psi) \to (X, \varphi)$ is a finite-to-one factor map, then we define
\[
Y_N(\pi) =\{(y_0,y_1, \cdots, y_N) \in Y^N \mid \pi(y_i)=\pi(y_j) \text{ for all } 0\leq i,j \leq N\}.
\]
The map $\psi(y_0,y_1,\dots,y_N):=(\psi(y_0),\psi(y_1),\dots,\psi(y_N)):Y_N(\pi) \to Y_N(\pi)$ makes $(Y_N(\pi),\psi)$ into a Smale space \cite[Proposition 2.4.4]{put}.
Suppose $(Y_N(\pi), \psi)$ is an $N$-fold fibre product of $(Y,\psi)$ over $(X,\varphi)$, and let $\check{y}_n$ denote deleting the $n$th coordinate of $Y_N(\pi)$, then we define $\delta_n:Y_N(\pi) \to Y_{N-1}(\pi)$ by
\begin{equation}\label{delete_n_Cech}
\delta_n(y_0,y_1,\cdots, y_N) = (y_0,y_1,\cdots,\check{y_n},\cdots,y_N).
\end{equation}
Putnam shows in \cite[Theorem 2.5.14]{put} that if $\pi:(Y,\psi) \to (X,\varphi)$ is s-bijective, then $\delta_n:(Y_N(\pi),\psi) \to (Y_{N-1}(\pi),\psi)$ is s-bijective for all $N \geq 1$ and $0\leq n \leq N$.

Bowen's Theorem \cite[Theorem 28]{Bow} for Smale spaces says that if $(X,\varphi)$ is a non-wandering Smale space, then there is a shift of finite type $(\Sigma,\sigma)$ and a factor map $\pi:(\Sigma,\sigma) \to (X,\varphi)$ such that $\pi$ is finite-to-one and one-to-one on a dense $G_\delta$ subset of $\Sigma$. In order to construct his homology theory, Putnam generalized Bowen's Theorem to show that every non-wandering Smale space $(X,\varphi)$ has an s/u-bijective pair as follows.

\begin{definition}[{\cite[Definition 2.6.2]{put}}]\label{s/u-bijective pair}
Suppose $(X,\varphi)$, $(Y,\psi)$, and $(Z,\zeta)$ are Smale spaces. The tuple $\pi=(Y,\psi,\pi_s,Z,\zeta,\pi_u)$ is called an s/u-bijective pair if
\begin{enumerate}
\item $\pi_s:(Y,\psi) \to (X,\varphi)$ is an s-bijective factor map,
\item $Y^u(y)$ is totally disconnected for all $y \in Y$,
\item $\pi_u:(Z,\zeta) \to (X,\varphi)$ is a u-bijective factor map, and
\item $Z^s(z)$ is totally disconnected for all $z \in Z$.
\end{enumerate}
\end{definition}

\begin{thm}[Putnam's generalization of Bowen's Theorem {\cite[Theorem 2.6.3]{put}}]
If $(X,\varphi)$ is a non wandering Smale space, then there exists an s/u-bijective pair for $(X,\varphi)$.
\end{thm}

Suppose $(X,\varphi)$ is a Smale space with an s/u-bijective pair $\pi=(Y,\psi,\pi_s,Z,\zeta,\pi_u)$. For $L,M \geq 0$ let
\[
\Sigma_{L,M}(\pi):=\{(y_0,\cdots,y_L,z_0,\cdots,z_M) \mid y_l \in Y, z_m \in Z, \pi_s(y_l)=\pi_u(z_m), 0 \leq l \leq L, 0 \leq m \leq M\}.
\]
Let $\sigma: \Sigma_{L,M}(\pi) \to \Sigma_{L,M}(\pi)$ be the map defined by
\[ 
\sigma_{L,M}(\pi)(y_0,\cdots,y_L,z_0,\cdots,z_M)=(\psi(y_0),\cdots,\psi(y_L),\zeta(z_0),\cdots,\zeta(z_M)).
\]
For $0 \leq l \leq L$ and $0 \leq m \leq M$, and define maps $\delta_{l,\,}:\Sigma_{L,M}(\pi) \to \Sigma_{L-1,M}(\pi)$ and $\delta_{\,,m}:\Sigma_{L,M}(\pi) \to \Sigma_{L,M-1}(\pi)$ by
\begin{align*}
\delta_{l,\,}(y_0,\cdots,y_L,z_0,\cdots,z_M)&=(y_0,\cdots,\check{y_l},\cdots,y_L,z_0,\cdots,z_M), \text{ and} \\
\delta_{\,,m}(y_0,\cdots,y_L,z_0,\cdots,z_M)&=(y_0,\cdots,y_L,z_0,\cdots,\check{z_m},\cdots,z_M).
\end{align*}
In \cite[Theorem 2.6.6]{put}, Putnam proves that if $\pi$ is an s/u-bijective pair for $(X,\varphi)$, then $(\Sigma_{L,M}(\pi),\sigma)$ is a shift of finite type for all $L,M \geq 0$. Moreover, $\delta_{l,\,}:\Sigma_{L,M}(\pi) \to \Sigma_{L-1,M}(\pi)$ is an s-bijective factor map and $\delta_{\,,m}:\Sigma_{L,M}(\pi) \to \Sigma_{L,M-1}(\pi)$ is a u-bijective factor map \cite[Theorem 2.6.13]{put}.

\subsection{Krieger's dimension groups for shifts of finite type}\label{DimensionGroups}

Wolfgang Krieger defined two abelian groups associated to a shift of finite type $(\Sigma,\sigma)$ in \cite{Kri}. In Putnam's homology theory, these are the homology groups for the zero dimensional Smale spaces. After formally defining Krieger's dimension groups, we give the properties of the dimension group required in this paper.

%

\begin{definition}[\cite{Kri,LM}] Given a graph $G$, the dimension groups of $(\Sigma_G, \sigma)$ are defined to be
\[
D^s(\Sigma_G, \sigma):= \lim_{\longrightarrow}\big( \Z^{|G^0|} \stackrel{A_G}{\longrightarrow} \Z^{|G^0|} \stackrel{A_G}{\longrightarrow} \Z^{|G^0|}
 \stackrel{A_G}{\longrightarrow} \big)
 \]
 and 
\[
D^u(\Sigma_G, \sigma):=  \lim_{\longrightarrow}\big( \Z^{|G^0|} \stackrel{A^t_G}{\longrightarrow} \Z^{|G^0|} \stackrel{A^t_G}{\longrightarrow} \Z^{|G^0|}
 \stackrel{A^t_G}{\longrightarrow}\big),\]
where $A_G$ is the adjacency matrix of $G$. 
We denote elements in the inductive limit as $[v, j]$ where $v \in \Z^{|G^0|}$ and $j\in \N$. Moreover, since $A_G$ and $A_G^t$ are positive maps and $(\Z^+)^{|G^0|}$ is a semigroup, there are dimension semigroups $D^s(\Sigma_G, \sigma)^+$ and $D^u(\Sigma_G, \sigma)^+$ that define a natural order structure on $D^s(\Sigma_G, \sigma)$ and $D^u(\Sigma_G, \sigma)$, respectively.
\end{definition}

\begin{definition} \label{mapFromMatrix}
Given graphs $G$ and $H$ and a nonnegative integer matrix $F\ne 0$ such that $F A_G = A_H F$. We let $F_*$ denote the map from $D^s(\Sigma_G, \sigma)$ to $D^s(\Sigma_H, \sigma)$ defined (at the level of inductive limits) via $[v,j] \mapsto [Fv,j]$.
In a similar way, let $F^*$ denote the map from $D^u(\Sigma_H, \sigma)$ to $D^u(\Sigma_G, \sigma)$ defined via $[v,j] \mapsto [F^tv,j]$. 
\end{definition}

Given graphs $G$ and $H$, a surjective graph homomorphism $\theta:G \to H$ is called a \emph{left-covering map} if $\theta:t^{-1}\{v\} \to t^{-1}\{\theta(v)\}$ is bijective for all $v \in G^0$ and is a \emph{right-covering map} if $\theta:i^{-1}\{v\} \to i^{-1}\{\theta(v)\}$ is bijective for all $v \in G^0$. Left-covering maps give s-bijective maps at the shift of finite type level. Similarly, right-covering maps give u-bijective maps at the shift of finite type level.

The dimension group is functorial with respect to s- and u-bijective maps. More precisely, if $\pi:(\Sigma_G, \sigma) \rightarrow (\Sigma_H, \sigma)$ is s-bijective, then, as discussed in \cite[Section 3.4]{put}, one obtains a map $\pi^s: D^s(\Sigma_G, \sigma) \rightarrow D^s(\Sigma_H, \sigma)$. On the other hand, if $\pi$ is u-bijective, then, as in \cite[Section 3.5]{put}, there is a map $\pi^{s*}: D^s(\Sigma_H, \sigma) \rightarrow D^s(\Sigma_G, \sigma)$. Similar results hold for $D^u( \: \cdot \: )$; the induced maps in this case are denoted by $\pi^u$ and $\pi^{u*}$, see \cite[Sections 3.4 and 3.5]{put} for further details. 

The next two propositions are special cases of Theorems 3.4.4 and 3.5.5 in \cite{put}, so the proofs are omitted.

\begin{prop} \label{leftcoveringMatrixMap}
Let $\theta:G \rightarrow H$ be a left-covering map. Then 
$\theta^s : D^s(\Sigma_G, \sigma) \rightarrow D^s(\Sigma_H,\sigma)$ is given by $[v, j] \mapsto [Dv, j]$, where $D$ is the matrix determined by $D_{IJ}=1$ if $\theta_0(J)=I$ and $D_{IJ}=0$ otherwise.
\end{prop}

\begin{prop} \label{rightcoveringMatrixMap}
Let $\theta: G \rightarrow H$ be a right-covering map. Then $\theta^{s*}: D^s(\Sigma_H, \sigma) \rightarrow D^s(\Sigma_G, \sigma)$ is given by $[v, j] \mapsto [Ev,j]$,
where $E$ is the matrix determined by $E_{IJ}=1$ if $\theta_0(I)=J$ and $E_{IJ}=0$ otherwise.
\end{prop}

The basic properties of the induced maps in the general case are summarized in the next proposition. Its proof is fairly straightforward using the functoriality results for the dimension group found in \cite[Sections 3.4 and 3.5]{put}.

\begin{prop} \label{ResOrdAndAutSFT}
Let $\pi: (\Sigma, \sigma) \rightarrow (\Sigma^{\prime}, \sigma)$ be a factor map between shifts of finite type. Then
\begin{enumerate}
\item If $\pi$ is s-bijective, then the induced maps $\pi^s: D^s(\Sigma, \sigma) \rightarrow D^s(\Sigma^{\prime}, \sigma)$ and $\pi^{u*}:D^u(\Sigma^{\prime}, \sigma) \rightarrow D^u(\Sigma, \sigma)$ preserve the order structure and intertwine the automorphisms induced from the relevant shift maps.
\item If $\pi$ is u-bijective, then the induced maps $\pi^u: D^u(\Sigma, \sigma) \rightarrow D^u(\Sigma^{\prime}, \sigma)$ and $\pi^{s*}:D^s(\Sigma^{\prime}, \sigma) \rightarrow D^s(\Sigma, \sigma)$ preserve the order structure and intertwine the automorphisms induced from the relevant shift maps.
\end{enumerate}
\end{prop}

\begin{lemma} \label{nToOneShiftLemma}
Suppose that $(\Sigma, \sigma)$ and $(\Sigma^{\prime}, \sigma)$ are shifts of finite type and $\eta: (\Sigma, \sigma) \rightarrow (\Sigma^{\prime}, \sigma)$ is a factor map, which is both s- and u-bijective and is $n$-to-one for some positive integer $n$. Then, $\eta^s\circ \eta^{s*}= n \cdot id_{D^s(\Sigma^{\prime}, \sigma)}$ and $\eta^u\circ \eta^{u*}= n \cdot id_{D^u(\Sigma^{\prime}, \sigma)}$.
\end{lemma}

\begin{proof}
We prove that $\eta^{s} \circ \eta^{s*}= n \cdot id_{D^s(\Sigma^{\prime}, \sigma)}$, the proof of the other equality is analogous. 

Let $e\in \Sigma^{\prime}$ and $E$ be a clopen set in $\Sigma^s(e)$. Then, \cite[Theorem 3.5.1]{put} implies that there is $k\ge 1$ and a family of subsets $\{ F_i \}_{i=1}^k$ such that 
\[
\eta^{-1}(E) = \bigcup_{i=1}^k F_i 
\]
and, for each $1\le i \neq j\le k$, no element in $F_i$ is stably equivalent to an element in $F_j$. Moreover, since $\eta$ is $n$-to-one
\[
\eta^{-1}(e)= \{ f_1, \ldots, f_n \},
\]
and since $\eta$ is s-bijective we also have $f_i \not\sim_s f_j$ for $i\neq j$. Since, for each $1\le i \le k$, all elements in $F_i$ are stably equivalent, we have that $k=n$ and, by possibly reordering, we can assume that $f_i \in F_i$ for each $i$.

Next, we show that $\eta(F_i)=E$ for each $i$. Fix $e^{\prime} \in E$, then using an argument similar to the one in previous paragraph, we obtain $f^{\prime}\in F_i$ such that $\eta(f^{\prime})=e^{\prime}$, and hence $\eta(F_i)=E$. Finally, using Theorems 3.4.1 and 3.5.1 in \cite{put}, we have that
\[ 
(\eta^s \circ \eta^{s*})([E]) = \eta^s ( [F_1] + \cdots + [F_n] ) = n \cdot [E]. \qedhere
\]
\end{proof}

If one assumes the Smale spaces in many of our results in this paper have certain recurrence properties, then the assumptions on the maps can be weakened. For example, we have the following result from \cite{DKWfunProPutHom}, which shows that the $n$-to-one hypothesis in the Lemma \ref{nToOneShiftLemma} is automatic if the shifts of finite type are irreducible.

\begin{thm}[{\cite[Theorem 3.3]{DKWfunProPutHom}}] \label{SUimpliesNtoOne}
Suppose $\pi: (Y, \psi) \to (X, \varphi)$ is a factor maps between irreducible Smale spaces which is both s- and u-bijective.  Then $\pi$ is $n$-to-$1$ for some positive integer $n$.
\end{thm}

\subsection{Putnam's homology theory for Smale spaces} \label{putHomRev}

Putnam's homology theory generalizes Krieger's dimension groups to all Smale spaces possessing an s/u-bijective pair. We first briefly define Putnam's homology theory and then give the results required in the sequel.

The complexes required to define Putnam's homology theory are defined in \cite[Section 5.1]{put}. For completeness we give a heuristic description, starting with \cite[Definition 5.1.1]{put}. 
Suppose $(X,\varphi)$ is a Smale space possessing an s/u-bijective pair $\pi$. For each $L,M \geq 0$, let $C^s(\pi)_{L,M}=D^s(\Sigma_{L,M}(\pi),\sigma)$ and let $C^s(\pi)_{L,M}=0$ if $L<0$ or $M<0$. Let $d^s(\pi)_{L,M}:C^s(\pi)_{L,M} \to C^s(\pi)_{L-1,M} \oplus C^s(\pi)_{L,M+1}$ be the maps defined by
\[
d^s(\pi)_{L,M}=\sum_{0\leq l\leq L} (-1)^l \delta_{l,\,}^s + \sum_{0\leq m\leq M+1} (-1)^{L+m} \delta_{\,,m}^{s *}.
\]
An analogous definition holds for unstable sets, see \cite[Definition 5.1.1~(2)]{put}.

For $N\in \N$, we let $S_N$ denote the associated permutation group. For each $L,M \geq 0$ the group $S_{L+1} \times S_{M+1}$ acts on $\Sigma_{L,M}(\pi)$, and the action commutes with the dynamics. For $L,M \geq 0$, \cite[Definition 5.1.5 (4)]{put} defines a complex $D_{\QQ,\AA}^s(\Sigma_{L,M}(\pi))$ by taking subgroups, quotients, and images of $D^s(\Sigma_{L,M}(\pi))$ with respect to the action of $S_{L+1} \times S_{M+1}$. Then \cite[Definition 5.1.7]{put} defines $C_{\QQ,\AA}^s(\pi)_{L,M}=D_{\QQ,\AA}^s(\Sigma_{L,M}(\pi))$ and lets $d_{\QQ,\AA}^s(\pi)_{L,M}$ be the boundary map on $C_{\QQ,\AA}^s(\pi)_{L,M}$ induced from $d^s(\pi)_{L,M}$. The complex $C_{\QQ,\AA}^s(\pi)$ has only a finite number of nonzero terms, as shown in \cite[Theorem 5.1.10]{put}. We then have the following definition.

\begin{definition}[{\cite[Definition 5.1.11]{put}}]\label{HomDefn}
Suppose $\pi$ is an s/u-bijective pair for $(X,\varphi)$. Then $H^s_*(\pi)$ is the homology of the double complex $(C_{\QQ,\AA}^s(\pi),d_{\QQ,\AA}^s(\pi))$ given by
\[
H^s_N(\pi):=\ker\Big(\bigoplus_{L-M=N} d_{\QQ,\AA}^s(\pi)_{L,M}\Big) \Big/ \im\Big(\bigoplus_{L-M=N+1} d_{\QQ,\AA}^s(\pi)_{L,M}\Big).
\]
The homology groups $H^u_*(\pi)$ are defined analogously.
\end{definition}

Putnam shows in \cite[Theorem 5.5.1]{put} that $H^s_*$ is independent of the s/u-bijective pair $\pi$, and hence defines a homology theory for $(X,\varphi)$ denoted $H^s_*(X,\varphi)$. 

Suppose there exists an s/u-bijective pair for $(X,\varphi)$, then the functor that associates a sequence of abelian groups $H^s_*(X,\varphi)$ to $(X,\varphi)$ is covariant for s-bijective factor maps and contravariant for u-bijective factor maps (in the same way as in Section \ref{DimensionGroups}). On the other hand the functor that associates a sequence of abelian groups $H^u_*(X,\varphi)$ to $(X,\varphi)$ is contravariant for s-bijective factor maps and covariant for u-bijective factor maps (in the same way as Section \ref{DimensionGroups}). For further details see \cite[Section 5.4]{put}. The notation for these functors is given in the next definition.

\begin{definition}
Let $(X,\varphi)$ and $(Y, \psi)$ be Smale spaces which each have an s/u-bijective pair. Also let $\pi: (X,\varphi) \rightarrow (Y, \psi)$ be a factor map. Then
\begin{enumerate}
\item If $\pi$ is s-bijective, then the induced maps on homology are denoted by $\pi^s: H^s(X,\varphi) \rightarrow H^s(Y, \psi)$ and $\pi^{u*}: H^u(Y, \psi ) \rightarrow H^u(X,\varphi)$.
\item If $\pi$ is u-bijective, then the induced maps on homology are denoted by $\pi^u: H^u(X,\varphi) \rightarrow H^u(Y, \psi)$ and $\pi^{s*}: H^s(Y, \psi ) \rightarrow H^s(X,\varphi)$.
\end{enumerate}
\end{definition}

\begin{prop} \label{bothSUmap}
Let $(X,\varphi)$ and $(X^{\prime}, \varphi^{\prime})$ be Smale spaces. If $\eta: (X,\varphi) \rightarrow (X^{\prime}, \varphi^{\prime})$ is both s- and u-bijective and is $n$-to-one for some positive integer $n$. Then, $\eta^{s} \circ \eta^{s*}=n \cdot id_{H^s(X^{\prime},\varphi^{\prime})}$ and $\eta^{u} \circ \eta^{u*} = n \cdot id_{H^u(X^{\prime}, \varphi^{\prime})}$. 
\end{prop}

Again, we note that when the Smale spaces are irreducible, the assumption that the map is $n$-to-one can be removed, since it is automatic in this case, see Theorem \ref{SUimpliesNtoOne} (which is \cite[Theorem 3.3]{DKWfunProPutHom}).

\begin{proof}
We prove that $\eta^{s} \circ \eta^{s*}= n \cdot id$, the proof of the other equality is similar. 

Let $\pi^{\prime}=(Y^{\prime}, \psi^{\prime}, \pi^{\prime}_s, Z^{\prime}, \zeta^{\prime}, \pi^{\prime}_u)$ be an s/u-bijective pair for $(X^{\prime}, \varphi^{\prime})$. Following \cite[Theorem 5.4.2]{put}, we have that 
\begin{align*}
(Y, \psi) & :=\fib{(X, \varphi)}{\eta}{\pi^{\prime}_s}{(Y^{\prime}, \psi^{\prime})} \\
(Z, \zeta) & := \fib{(X, \varphi)}{\eta}{\pi^{\prime}_u}{(Z^{\prime}, \zeta^{\prime})} 
\end{align*}
such that $\pi_s:=p_1: \fib{(X, \varphi)}{\eta}{\pi^{\prime}_s}{(Y^{\prime}, \psi^{\prime})} \rightarrow (X,\varphi)$ and $\pi_u:=p_1: \fib{(X, \varphi)}{\eta}{\pi^{\prime}_u}{(Z^{\prime}, \zeta^{\prime})} \rightarrow (X,\varphi)$ are an s/u-bijective pair for $(X,\varphi)$; we denote it by $\pi$. Moreover, we have a commutative diagram
\[
\begin{CD}
(Y,\psi) @>>\pi_s> (X,\varphi) @<<\pi_u< (Z,\zeta)\\
@VV \eta_Y V @VV \eta_X V @VV \eta_Z V\\
(Y',\psi') @>> \pi'_s > (X',\varphi') @<<\pi'_u< (Z',\zeta')\\
\end{CD}
\]
where $\eta_Y:=p_2:\fib{(X, \varphi)}{\eta}{\pi^{\prime}_s}{(Y^{\prime}, \psi^{\prime})} \rightarrow (Y^{\prime}, \psi^{\prime})$ and $\eta_Z:=p_2: \fib{(X, \varphi)}{\eta}{\pi^{\prime}_u}{(Z^{\prime}, \zeta^{\prime})} \rightarrow (Z^{\prime}, \zeta^{\prime})$.

It follows from standard properties of fibre products (compare with \cite[Proposition 8.3.3]{LM}) that $\eta_Y$ and $\eta_Z$ are s-bijective, u-bijective and $n$-to-one. For each $L\ge 0$ and $M\ge 0$, in \cite[Section 5.4]{put} Putnam defines a map
\[
\eta_{\Sigma_{L,M}(\pi)}: \Sigma_{L,M}(\pi) \rightarrow \Sigma_{L,M}(\pi^{\prime}).
\]
Again, using properties of fibre products it follows that $\eta_{\Sigma_{L,M}(\pi)}$ is also s-bijective, u-bijective, and $n$-to-one.

Lemma \ref{nToOneShiftLemma} implies that, for each $L\ge 0$ and $M\ge 0$, the map 
\[ 
(\eta_{\Sigma_{L,M}(\pi)})^s \circ  (\eta_{\Sigma_{L,M}(\pi)})^{s*}:D^s(\Sigma_{L,M}(\pi^{\prime})) \rightarrow D^s(\Sigma_{L,M}(\pi^{\prime}))  
\]
is equal to $n \cdot id_{D^s(\Sigma_{L,M}(\pi))}$. Since the map on homology is induced from this map, it follows that $\eta^s \circ \eta^{s*}=n \cdot id_{H^s(X^{\prime},\varphi^{\prime})}$.
\end{proof}
\begin{prop} Putnam's homology theory has the following basic invariance properties:
\begin{enumerate}
\item $(H^s_*(X,\varphi), (\varphi^{-1})^s)$ and $(H^u_*(X,\varphi), \varphi^u)$ are conjugacy invariants. 
\item For any $l\ge 1$, $H^s_*(X,\varphi) \cong H^s_*(X, \varphi^l)$ and $H^u_*(X,\varphi)\cong H^u_*(X,\varphi^l)$, and these isomorphisms are natural under the identifications $(\varphi^{l})^s=(\varphi^{s})^l$ and $(\varphi^{l})^u=(\varphi^{u})^l$. 
\end{enumerate}
\end{prop}

We now have the following proposition, whose proof follows directly from the results in \cite[Sections 5.4 and 5.5]{put}.

\begin{prop}
Let $(X,\varphi)$ and $(Y,\psi)$ be nonwandering Smale spaces and $\pi: (X,\varphi) \to (Y,\psi) $ be a factor map. Then
\begin{enumerate}
\item If $\pi$ is s-bijective, then the induced maps $\pi^s: H^s(X,\varphi) \rightarrow H^s(Y,\psi)$ and $\pi^{u*}:H^u(Y,\psi) \rightarrow H^u(X,\varphi)$ intertwine the automorphisms induced from $\varphi$ and $\psi$.
\item If $\pi$ is u-bijective, then the induced maps $\pi^u: H^u(X, \varphi) \rightarrow H^u(Y, \psi)$ and $\pi^{s*}:H^s(Y,\psi) \rightarrow H^s(X,\varphi)$ intertwine the automorphisms induced from $\varphi$ and $\psi$.
\end{enumerate}
\end{prop}
\begin{remark}
At present the definition of an order on $H^s_0(X,\varphi)$ respectively $H^u_0(Y,\varphi)$ has not be defined. However, it is likely that such orders exist, \cite[Section 8.1]{put}. If the maps induced from s- and u-bijective maps respect these orderings, then ``ordered" can be added to the previous theorem in the same manner as in Proposition \ref{ResOrdAndAutSFT}.
\end{remark}

\subsection{Almost one-to-one factor maps and ${\rm Per}(X, \varphi)$}
\begin{definition}
Suppose $(X, \varphi)$ is a Smale space, $(Y, \psi)$ is an irreducible Smale space, and $\pi : (X, \varphi) \rightarrow (Y, \psi)$ is a finite-to-one factor map. Then 
\[
{\rm deg}(\pi):= {\rm min}\{ {\rm card}\left( \pi^{-1}\{ y \} \right) \: | \: y \in Y \}
\]
where ${\rm card}(A)$ denotes the cardinality of the set $A$. Moreover, $\pi$ is called \emph{almost one-to-one} if ${\rm deg}(\pi)=1$.
\end{definition}
\begin{definition}
Suppose $(X,\varphi)$ is an irreducible Smale space. Then 
\[ {\rm Per}(X,\varphi):={\rm gcd}\{ n \in \N \: | \: {\rm per}_n(X,\varphi)\neq \varnothing  \}
\] 
where ${\rm per}_n(X,\varphi)$ denotes the set of periodic points of period $n$ in $(X,\varphi)$.
\end{definition}

\begin{lemma} \label{perLemma}
Suppose that $\pi:(X,\varphi) \rightarrow (Y,\psi)$ is an almost one-to-one factor map between irreducible Smale spaces. Then, ${\rm Per}(X, \varphi)={\rm Per}(Y, \psi)$.
\end{lemma}
\begin{proof}
To begin, we show that ${\rm Per}(Y, \psi)$ divides ${\rm Per}(X, \varphi)$. To this end, suppose that $y \in Y$ is a periodic point of least period $l$. Since $\pi$ is onto, there exists $x \in X$ such $\pi(x)=y$. Hence, for any $i \in \mathbb{N}$,
\[ 
\pi( \varphi^{il}( x))= \psi^{il}(\pi(x))=\psi^{il}(y)=y.
\]
Since $\pi$ is finite-to-one, the pigeonhole principle implies that $x$ is periodic. Furthermore, if $k$ is the least period of $x$, then 
\[
\psi^k(y)=\psi^k(\pi(x))=\pi(\varphi^k(x))=\pi(x)=y,
\] 
from which it follows that $l$ divides $k$, and hence that ${\rm Per}(Y, \psi)$ divides ${\rm Per}(X, \varphi)$. Let $q$ be the quotient.  We will show that $q = 1$.


Smale's decomposition theorem \cite[Theorem 2.1.13]{put} implies that $X=X_1 \dot{\cup} X_2 \dot{\cup} \cdots \dot{\cup} X_{{\rm Per}(X, \varphi)}$ where each $X_i$ is clopen and $(X_i, \varphi^{{\rm Per}(X, \varphi)}|_{X_i})$ is mixing for each $i$. Similarly, we have that $Y=Y_1 \dot{\cup} Y_2 \dot{\cup} \cdots \dot{\cup} Y_{{\rm Per}(Y, \psi)}$ where each $Y_j$ is clopen and $(Y_j, \varphi^{{\rm Per}(Y, \psi)}|_{Y_j})$ is mixing for each $j$.  Since $\pi$ is almost one-to-one, it is routine to see that the set of points with unique preimage is dense in Y.  We proceed by assuming that $q>1$ and we arrive at a contradiction by showing this assumption leads to the existence of an open set in which every point has at least two preimages.

Suppose $q >1$ and let $x \in X_i$. Then $\pi(x)\in Y$ and hence $\pi(x)\in Y_j$ for some $j$. Now suppose $x^{\prime} \in X_i$, we will show that $\pi(x^{\prime}) \in Y_j$. Mixing implies that $X^s(x)$ is dense in $X_i$, so there is a sequence $\{ x_n\}_{n\in \N}$ which converges to $x^{\prime}$ and $x_n \sim_s x$ for each $n$. We have that each $\pi(x_n) \in Y_j$ and they converge to $\pi(x^{\prime})$. Since $Y_j$ is closed, $\pi(x^{\prime})\in Y_j$. To summarize, we have showed that there is unique $j$ such that $\pi(X_i) \subseteq Y_j$.  Since ${\rm Per}(Y, \psi)$ is the least integer, $k$, such that $\psi^k(Y_j) \cap Y_j \neq \emptyset $ it follows that the components of $X$ which map into $Y_j$ are precisely those components $\varphi^{m\cdot  {\rm Per}(Y, \psi)}(X_i)$ for $0 \leq m < q$. 

If $\pi(X_i) = Y_j$, then $\pi(\varphi^{{\rm Per}(Y, \psi)}(X_i)) = \psi^{{\rm Per}(Y, \psi)}(\pi(X_i)) = Y_j$ and $Y_j$ is the open set which gives the contradiction.

If $\pi(X_i) \neq Y_j$, let $L$ be the greatest integer such that $\cup_{l=0}^L \pi(\varphi^{l \cdot  {\rm Per}(Y, \psi)}(X_i)) \neq Y_j$ and let $\tilde{X} = \cup_{l=0}^L \varphi^{l \cdot  {\rm Per}(Y, \psi)}(X_i)$. $\tilde{X}$ is compact, and hence so is $\pi(\tilde{X})$.  It follows that $U = \pi(\tilde{X})^c \subseteq \pi(\varphi^{(L+1) \cdot  {\rm Per}(Y, \psi)}(X_i))$ is open and non-empty. Since $\pi(X_i) \neq Y_j$, it follows that $\pi(\varphi^{(L+1) \cdot  {\rm Per}(Y, \psi)}(X_i)) \neq Y_j$, so $V = \pi(\varphi^{(L+1) \cdot  {\rm Per}(Y, \psi)}(X_i))^c \subseteq \pi(\tilde{X})$ is open and non-empty. The fact that $\psi^{{\rm Per}(Y, \psi)}$ is mixing on $Y_j$ and the fact that ${\rm Per}(X, \varphi) = q \cdot {\rm Per}(Y, \psi)$ then imply that there exists $M$ such that 
\[ 
\psi^{M \cdot {\rm Per}(X, \varphi)}(U) \cap V \neq \emptyset. 
\]
Recall that
\[
\psi^{M \cdot {\rm Per}(X, \varphi)}(U) \subseteq \pi(\varphi^{M \cdot {\rm Per}(X, \varphi)}\varphi^{(L+1) \cdot  {\rm Per}(Y, \psi)}(X_i)) = \pi(\varphi^{(L+1) \cdot  {\rm Per}(Y, \psi)}(X_i)).
\]
so we have an open set contained in $\pi(\tilde{X}) \cap \pi(\varphi^{(L+1)\cdot{\rm Per}(Y, \psi)}(X_i))$. Now, each point $y$ in this open set has at least two preimages, and as above this is a contradiction to $q >1$ and therefore we must have $q = 1$ and ${\rm Per}(Y, \psi) = {\rm Per}(X, \varphi)$.

\end{proof}

\section{Finite equivalence and almost conjugacy}\label{Sec:Finite_equiv}
The definitions of finite equivalence and almost conjugacy for shifts of finite type (see \cite[Definitions 8.3.1 and 9.3.1]{LM}) naturally generalize to other dynamical systems. Among the first instances of this type of generalization go back to Parry \cite{Par} and Adler and Marcus \cite{AM}. However, the definition of a Smale space appeared after \cite{Par} and \cite{AM}, so we give a self-contained treatment of finite equivalence and almost conjugacy for Smale spaces. However, we note that the results of this section are not new since they follow easily from  \cite{AM, Par}, and are included for completeness and to emphasize the power of combining results for shifts of finite type and Bowen's theorem. In contrast, we will see in Section \ref{sec:correspondences} that such proofs cannot be used to obtain results on the existence of correspondences for general Smale spaces.

\begin{definition}[{\cite[p.89]{Par}}, also see {\cite[Definition 8.3.1]{LM}}] \label{def:finite-equiv}
Let $(X,\varphi)$ and $(Y,\psi)$ be Smale spaces. Then a \emph{finite equivalence} between $(X,\varphi)$ and $(Y, \psi)$ is the following diagram:
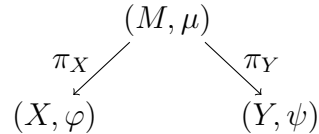
\[
\begin{tikzpicture}
\Corr{0}{0}{(X,\varphi)}{\pi_X}{(M,\mu)}{\pi_Y}{(Y,\psi)}
\end{tikzpicture}
\]
where
\begin{enumerate}
\item $(M,\mu)$ is a Smale space;
\item $\pi_X$ and $\pi_Y$ are finite-to-one factor maps.
\end{enumerate}
\end{definition} 

\begin{prop} [{\cite[p.89]{Par}}, also see {\cite[Proposition 8.3.4]{LM}}] \label{finiteEquIsAnEqu} 
Finite equivalence for Smale spaces is an equivalence relation.
\end{prop}
\begin{proof}
By \cite[Theorem 2.4.2]{put}, the fibre product of two Smale spaces over factor maps is again a Smale space. Now the proof of \cite[Proposition 8.3.4]{LM} generalizes to Smale spaces with finite-to-one factor maps with only minor changes.
\end{proof}

\begin{prop}[cf. {\cite[Theorem 1 and Lemma 2]{Par}}, also see {\cite[Theorem 8.3.7]{LM}}] \label{finiteEqualEntropy}
Two  irreducible Smale spaces are finitely equivalent if and only if they have the same entropy, and the forward implication holds for general Smale spaces.
\end{prop}
\begin{proof}
Since finite-to-one factor maps preserve entropy \cite[Problem 8.2.6]{HK}, finitely equivalent Smale spaces have the same entropy. 

For the other implication, let $(X,\varphi)$ and $(Y,\psi)$ be irreducible Smale spaces with equal entropy. Bowen's theorem \cite{Bow} implies there exist shifts of finite type $(\Sigma_X,\sigma)$ and $(\Sigma_Y,\sigma)$ and finite-to-one factor maps 
\[ \eta_X : (\Sigma_X, \sigma) \rightarrow (X, \varphi) \quad \text{ and } \quad \eta_Y: (\Sigma_Y, \sigma) \rightarrow (Y,\psi). \]
Since $(X,\varphi)$ and $(Y,\psi)$ are irreducible, we can also choose $(\Sigma_X,\sigma)$ and $(\Sigma_Y, \sigma)$ to be irreducible. Furthermore, finite-to-one factor maps preserve entropy \cite[Problem 8.2.6]{HK}, so $(\Sigma_X,\sigma)$ and $(\Sigma_Y, \sigma)$ have the same entropy. Hence, \cite[Theorem 8.3.7]{LM} implies  $(\Sigma_X,\sigma)$ and $(\Sigma_Y, \sigma)$ are finitely equivalent as shifts of finite type, see \cite[Definition 8.3.1]{LM}. That is, there exists a diagram
\[
\begin{tikzpicture}
\Corr{0}{0}{(\Sigma_X,\sigma)}{\pi_{\Sigma_X}}{(\Sigma_M,\sigma)}{\pi_{\Sigma_Y}}{(\Sigma_Y,\sigma)}
\end{tikzpicture}
\]
where
\begin{enumerate}
\item  $(\Sigma_M, \sigma)$ is a shift of finite type; 
\item $\pi_{\Sigma_X}$ and $\pi_{\Sigma_Y}$ are finite-to-one factor maps.
\end{enumerate}
Thus the following diagram is a finite equivalence:
\[
\begin{tikzpicture}
\Corr{0}{0}{(X,\varphi)}{\eta_X \circ \pi_{\Sigma_X}}{(\Sigma_M,\sigma)}{\eta_Y \circ \pi_{\Sigma_Y}}{(Y,\psi)}
\end{tikzpicture} \qedhere
\]
\end{proof}

Next we consider almost conjugacy in the setting of irreducible Smale spaces.

\begin{definition}[{\cite[Definition 2.16]{AM}}, also see {\cite[Definition 9.3.1]{LM}}] \label{almostConj}
Let $(X,\varphi)$ and $(Y,\psi)$ be irreducible Smale spaces. Then an \emph{almost conjugacy} between $(X,\varphi)$ and $(Y, \psi)$ is a diagram
\[
\begin{tikzpicture}
\Corr{0}{0}{(X,\varphi)}{\pi_X}{(M,\mu)}{\pi_Y}{(Y,\psi)}
\end{tikzpicture}
\]
such that
\begin{enumerate}
\item $(M,\mu)$ is a Smale space;
\item $\pi_X$ and $\pi_Y$ are finite-to-one almost one-to-one factor maps.
\end{enumerate}
\end{definition} 

\begin{prop} [{\cite[Theorem 2.17]{AM}}]
Almost conjugacy for irreducible Smale spaces is an equivalence relation.
\end{prop}

\begin{prop}[cf. {\cite[Corollary 13.2]{AM}}, also see {\cite[Theorem 9.3.2]{LM}}]
Two  irreducible Smale spaces, $(X,\varphi)$ and $(Y,\psi)$, are almost conjugate if and only if they have equal entropy and ${\rm Per}(X, \varphi)={\rm Per}(Y, \psi)$.
\end{prop}

\begin{proof}
Suppose $(X,\varphi)$ and $(Y,\psi)$ are almost conjugate. Since any almost conjugacy is a finite equivalence, $(X,\varphi)$ and $(Y,\psi)$ have equal entropy by Proposition \ref{finiteEqualEntropy}. In addition, Lemma \ref{perLemma} implies that ${\rm Per}(X, \varphi)={\rm Per}(Y, \psi)$.

Now suppose $(X,\varphi)$ and $(Y,\psi)$ have equal entropy and ${\rm Per}(X, \varphi)={\rm Per}(Y, \psi)$. Then Bowen's theorem implies there are irreducible shifts of finite type $(\Sigma_X, \sigma)$  and $(\Sigma_Y, \sigma)$ and almost one-to-one factor maps $\pi_X$ and $\pi_Y$ such that
\[
\pi_X : (\Sigma_X, \sigma) \rightarrow (X, \varphi) \quad \hbox{ and } \quad \pi_Y: (\Sigma_Y, \sigma) \rightarrow (Y, \psi).
\]
Moreover, since finite-to-one factor maps preserve entropy \cite[Problem 8.2.6]{HK}, $(\Sigma_X, \sigma)$ and $(\Sigma_Y, \sigma)$ have the same entropy as $(X, \varphi)$ and $(Y,\psi)$, respectively. Thus $(\Sigma_X, \sigma)$ and $(\Sigma_Y, \sigma)$ have equal entropy as well. Now Lemma \ref{perLemma} implies that 
\[
{\rm Per}(\Sigma_X, \sigma) ={\rm Per}(X,\varphi) = {\rm Per}(Y,\psi)= {\rm Per}(\Sigma_Y, \sigma), 
\]
and  \cite[Theorem 9.3.2]{LM} implies that $(\Sigma_X, \sigma)$ and $(\Sigma_Y, \sigma)$ are almost conjugate. It follows that $(X,\varphi)$ and $(Y,\psi)$ are also almost conjugate (the details are similar to the proof of Proposition \ref{finiteEqualEntropy}).
\end{proof}

\section{Correspondences}\label{sec:correspondences}
While the results in the previous section on finite equivalence and almost conjugacy for Smale spaces naturally generalize from those for shifts of finite type, our main goal in this paper is the construction of ``generalized morphisms" between Smale spaces. In particular, from such a morphism we would like to obtain a map at the level of Putnam's homology theory for Smale spaces; a finite equivalence or an almost conjugacy does not naturally give such a map in general. We must put further conditions on the maps in a finite equivalence to get an induced map.

\begin{definition}\label{correspondence}
Suppose $(X,\varphi)$ and $(Y,\psi)$ are Smale spaces. A {\em correspondence} from $(X,\varphi)$ to $(Y,\psi)$ consists of:
\begin{enumerate}
\item a Smale space $(M,\mu)$;
\item a $u$-bijective map $\pi_u:(M,\mu) \to (X,\varphi)$; and
\item an $s$-bijective map $\pi_s:(M,\mu) \to (Y,\psi)$.
\end{enumerate}
We write correspondences as
\[
\begin{tikzpicture}
\Corr{0}{0}{(X,\varphi)}{\pi_u}{(M,\mu)}{\pi_s}{(Y,\psi)}
\end{tikzpicture}
\]
\end{definition}

Correspondences satisfy both the requirements mentioned at the beginning of this section. They naturally give rise to maps on homology and give a notion of a generalized morphism between Smale spaces, in the sense that two correspondences can be composed. Using the notation of the definition of a correspondence, the induced maps are defined by
\begin{align*}
\pi_s^s \circ \pi_u^{s *}&: H^s(X,\varphi) \to H^s(Y,\psi) \quad \text{ and} \\
\pi_u^u \circ \pi_s^{u *}&: H^u(Y,\psi) \to H^u(X,\varphi).
\end{align*}

\begin{deflem}
The {\em composition} of correspondences
\[
\begin{tikzpicture}
\Corr{0}{0}{(X,\varphi)}{\pi_{u}}{(M,\mu)}{\pi_{s}}{(Y,\psi)}
\node at (4.35,1) { and };
\Corr{6}{0}{(Y,\psi)}{\pi_{u}'}{(M',\mu')}{\pi_{s}'}{(Z,\eta)}
\end{tikzpicture}
\]
is defined to be the correspondence
\[
\begin{tikzpicture}
\Corr{0}{0}{(X,\varphi)}{(\pi_{u} \circ P_1)}{\fib{(M,\mu)}{\pi_{s}}{\pi_{u}'}{(M',\mu')}}{(\pi_{s}' \circ P_2)}{(Z,\eta)}
\end{tikzpicture}
\] 
where $\fib{(M,\mu)}{\pi_{s}}{\pi_{u}'}{(M',\mu')}$ is the fibre product and $P_1$ and $P_2$ are the projection maps (see Definition \ref{defOfFibPro}).
\end{deflem}

\begin{proof}
Our first goal is to show that composition of correspondences is compatible at the level of the maps induced on homology; that is, we show that
\begin{equation}\label{composition_hom}
(\pi_s' \circ P_2)^s \circ (\pi_u \circ P_1)^{s *} = (\pi_s'^s \circ \pi_u'^{s *})\circ (\pi_s^s \circ \pi_u^{s *})
\end{equation}
To show this, consider the diagram
\[
\begin{tikzpicture}
\Corr{0}{0}{(X,\varphi)}{\pi_{u}}{(M,\mu)}{\pi_{s}}{(Y,\psi)}
\Corr{3}{0}{}{\pi_{u}'}{(M',\mu')}{\pi_{s}'}{(Z,\eta)}
\Corr{1.5}{1.3}{}{P_1}{(M,\mu) _{\pi_{s}}\hspace{-0.1cm}\times_{\pi_{u}'} \hspace{-0.1cm}(M',\mu')}{P_2}{}
\end{tikzpicture}
\]
By \cite[Theorem 5.4.1]{put}, we have
\begin{equation}\label{comp_hom1}
(\pi_s' \circ P_2)^s \circ (\pi_u \circ P_1)^{s *} = (\pi_s')^s \circ (P_2)^s \circ (P_1)^{s *} \circ (\pi_u)^{s *}
\end{equation}
By \cite[Theorem 5.1]{DKWfunProPutHom} (see Remark \ref{nonWanCon} below), we have
\[
(P_2)^s \circ (P_1)^{s *}=(\pi_u')^{s *} \circ (\pi_s)^s,
\]
and substituting this into the right hand side of \eqref{comp_hom1} gives \eqref{composition_hom}, which is the desired result.
\end{proof}

\begin{remark} \label{nonWanCon}
The statement of Theorem 5.1 in \cite{DKWfunProPutHom} contains the assumption that the fibre product space is nonwandering; this assumption is superfluous. In particular, it follows from the proof of \cite[Theorem 5.1]{DKWfunProPutHom} that the class of Smale spaces which have an s/u-bijective pair is closed under taking fibre products over s- or u-bijective maps.
\end{remark}

\subsection{Correspondences for subshifts of finite type} \label{corForSFT}

For shifts of finite type, the existence of correspondences and finite equivalences are equivalent. In particular, we add correspondences to the list of equivalent statements in \cite[Theorem 8.3.8]{LM} (which first appeared as Theorem 1 and Lemma 2 in \cite{Par}) as follows.

\begin{thm}[reformulation of {\cite[Theorem 8.3.8]{LM}}] \label{838LMinCor}
Let $G$ and $H$ be irreducible graphs. Then the following are equivalent
\begin{enumerate}
\item $(\Sigma_G, \sigma)$ and $(\Sigma_H, \sigma)$ are finitely equivalent;
\item $(\Sigma_G, \sigma)$ and $(\Sigma_H, \sigma)$ have equal entropy;
\item There exists a nonnegative integer matrix $F\ne 0$ such that $F A_G  = A_H F$; and
\item There is a correspondence from $(\Sigma_G, \sigma)$ to $(\Sigma_H, \sigma)$.
\end{enumerate}
\end{thm}

\begin{proof}
The proof of \cite[Theorem 8.3.8]{LM} shows the equivalence of (1) -- (3) and that (3) induces a finite equivalence between $(\Sigma_G, \sigma)$ and $(\Sigma_H, \sigma)$ where the maps $\pi_X$ and $\pi_Y$ can be taken to be right-covering and left-covering maps, respectively. Then (3) implies (4) follows from \cite[Theorem 2.5.17]{put}, which asserts that right-covering maps are $u$-bijective and left covering maps are $s$-bijective. Finally (4) implies (1) follows from the definitions and \cite[Theorem 2.5.3]{put}, which implies that s- and u- bijective maps are finite-to-one.
\end{proof}

\begin{remark}
A comment on the proof of \cite[Theorem 8.3.8]{LM} is in order. To define a finite equivalence, a Smale space $(M,\mu)$ was constructed in the proof. This Smale space is also required in the definition of a correspondence. We note that no assumptions are required on $(M,\mu)$. On the one hand, since $(\Sigma_G, \sigma)$ and $(\Sigma_H, \sigma)$ are irreducible, $(M,\mu)$ can be taken to be irreducible if desired. On the other hand, the process of restriction to an irreducible component of maximal entropy (see \cite[8.3.6]{LM}) does not respect the map induced on homology; explicit examples can be produced using the full two shift. Secondly, we have reversed the edge directions in the graph $M$ implicitly constructed in the proof of Theorem \ref{838LMinCor}. The $F$-edges in \cite{LM} have initial vertex in $G$ and terminal vertex in $H$, ours have initial vertex in $H$ and terminal vertex in $G$. This is because our matrix $F$ appears on the left of $A_G$ while theirs appears on the right.
\end{remark}

As we have mentioned above, we are interested in the map induced by a correspondence at the level of Putnam's homology theory for Smale spaces. For the following theorem, recall that the homology of a shift of finite type is its dimension group.

\begin{thm} \label{corConFromMatrix}
Let $G$ and $H$ be irreducible graphs and $F$ be a nonzero, nonnegative integer matrix such that $F A_G = A_H F$. Then, using the notation of Definition \ref{mapFromMatrix}, there exists a correspondence 
\begin{equation}\label{corresp_F}
\begin{tikzpicture}
\Corr{0}{0}{(\Sigma_G,\sigma)}{\pi_u}{(\Sigma_M,\sigma)}{\pi_s}{(\Sigma_H,\sigma)}
\end{tikzpicture}
\end{equation}
such that 
\[
\pi_s^s \circ \pi_u^{s *} =  F_* \text{ and } 
\pi_u^u \circ \pi_s^{u *}=  F^*.
\]
\end{thm} 

\begin{proof}
In the proof of Theorem \ref{838LMinCor}, given $F$ such that $FA_G=A_H F$, we used \cite[Theorem 8.3.8]{LM} to construct the correspondence \eqref{corresp_F} where $\pi_u$ and $\pi_s$ are induced from right-covering and left-covering maps, respectively. However, there are choices involved in the construction of both maps $\pi_u$ and $\pi_s$ as well as in the graph $M$, but any choice will lead to the desired result. To see this, note that the choices determine the edges of $M$, but not its vertices. Our proof only involves the vertices and the maps on the vertices of the relevant graphs.

Given the left and right covering maps, Propositions \ref{leftcoveringMatrixMap} and \ref{rightcoveringMatrixMap} (also see \cite[Definitions 2.4.11 and 8.2.4]{LM}) define $0$-$1$-matrices $D$ and $E$, associated to $\pi_s$ and $\pi_u$, respectively. Let an $F$-edge $e$ be a vertex in $M$, then  $E_{eJ}$ is one if $t(e) = J$ and zero otherwise and $D_{Ie}$ is one if $i(e) = I$ and zero otherwise. 
\[
(DE)_{IJ} = \sum_{e\in M^0} D_{Ie} \cdot E_{eJ} = | \{ F\text{-edges from }I \text{ to }J\}|=F_{IJ}.
\]
This shows that $\pi_s^s \circ \pi_u^{s *} =  F_*$.  The proof that $\pi_u^u \circ \pi_s^{u *}=  F^*$ follows by taking transposes of $D$, $E$, and $F$.
\end{proof}

The above proof relies heavily on the construction of a correspondence from an intertwining matrix as described in the proof of \cite[Theorem 8.3.8]{LM}.  As noted, there are ``choices'' in this construction which lead to different correspondences.  The fact that these choices do not change the resulting map on homology indicates that in some sense the different correspondences are ``equivalent''.  In the following section we will define several notions of equivalent correspondences and we will say more about the current situation in Example \ref{LMconstruction}.

\section{Equivalences between correspondences}\label{sec:equivalences}

In this section we introduce several notions of equivalence of correspondences. In Section \ref{equiv_corresp} we define four notions of equivalence on correspondences, each leading to a notion of invertibility. Section \ref{equiv_Smale} shows that each notion of invertibility at the correspondence level leads to an equivalence relation at the Smale space level. In Section \ref{equiv_exist} we consider equivalences based on only the existence of correspondences rather than ones associated to correspondences that induce particular maps on homology. In section \ref{sec:dimension} we study the implications of the existence of a correspondence on the dimensions of the Smale spaces involved.

\subsection{Equivalences of correspondences}\label{equiv_corresp}

We present four notions of equivalence for correspondences. 
\begin{definition}\label{Iso}
We say that correspondences
\[
\begin{tikzpicture}
\Corr{0}{0}{(X,\varphi)}{\pi_{u1}}{(M_1,\mu_1)}{\pi_{s1}}{(Y,\psi)}
\node at (4.35,1) { and };
\Corr{6}{0}{(X,\varphi)}{\pi_{u2}}{(M_2,\mu_2)}{\pi_{s2}}{(Y,\psi)}
\end{tikzpicture}
\]
are {\em isomorphic} if there exists a conjugacy $\Theta: (M_1, \mu_1) \to {(M_2,\mu_2)}$ such that the following diagram commutes.
\[
\begin{tikzpicture}
\Correq{0}{0}{(X,\varphi)}{(Y,\psi)}{\pi_{u1}}{\pi_{s1}}{\pi_{u2}}{\pi_{s2}}
\begin{scope}[xshift=0cm,yshift=1.8cm]
\node at (0,0) {$(M_1,\mu_1)$};
\node at (3,0) {$(M_2,\mu_2)$};
\draw[->] (.9,0) -- node[above] {$\Theta$\,} (2.1,0);
\end{scope}
\end{tikzpicture}
\]
\end{definition}

\begin{definition}\label{ratIso}
We say that correspondences
\[
\begin{tikzpicture}
\Corr{0}{0}{(X,\varphi)}{\pi_{u1}}{(M_1,\mu_1)}{\pi_{s1}}{(Y,\psi)}
\node at (4.35,1) { and };
\Corr{6}{0}{(X,\varphi)}{\pi_{u2}}{(M_2,\mu_2)}{\pi_{s2}}{(Y,\psi)}
\end{tikzpicture}
\]
are {\em rationally isomorphic} if there exists a correspondence 
\[
\begin{tikzpicture}
\Corr{0}{0}{(M_1,\mu_1)}{\theta_1}{(M,\mu)}{\theta_2}{(M_2,\mu_2)}
\end{tikzpicture}
\]
such that $\theta_1: (M, \mu) \to {(M_1,\mu_1)}$ is $m$-to-one and $\theta_2: (M, \mu) \to {(M_2,\mu_2)}$ is $n$-to-one, both $\theta_1$ and $\theta_2$ are s- and u-bijective, and  the following diagram commutes.
\[
\begin{tikzpicture}
\Corr{0}{1.8}{(M_1,\mu_1)}{\theta_1}{(M,\mu)}{\theta_2}{(M_2,\mu_2)};
\Correq{0}{0}{(X,\varphi)}{(Y,\psi)}{\pi_{u1}}{\pi_{s1}}{\pi_{u2}}{\pi_{s2}}
\end{tikzpicture}
\]
\end{definition}

\begin{definition}\label{HEquiv}
We say that correspondences
\[
\begin{tikzpicture}
\Corr{0}{0}{(X,\varphi)}{\pi_{u1}}{(M_1,\mu_1)}{\pi_{s1}}{(Y,\psi)}
\node at (4.35,1) { and };
\Corr{6}{0}{(X,\varphi)}{\pi_{u2}}{(M_2,\mu_2)}{\pi_{s2}}{(Y,\psi)}
\end{tikzpicture}
\]
are {\em $H$-equivalent} if 
\begin{enumerate}
\item $\pi_{s1}^s \circ \pi_{u1}^{s*} = \pi_{s2}^s \circ \pi_{u2}^{s*}$ and
\item $\pi_{u1}^u \circ \pi_{s1}^{u*} = \pi_{u2}^u \circ \pi_{s2}^{u*}$.
\end{enumerate}
\end{definition}

\begin{definition}\label{ratHEquiv}
We say that correspondences
\[
\begin{tikzpicture}
\Corr{0}{0}{(X,\varphi)}{\pi_{u1}}{(M_1,\mu_1)}{\pi_{s1}}{(Y,\psi)}
\node at (4.35,1) { and };
\Corr{6}{0}{(X,\varphi)}{\pi_{u2}}{(M_2,\mu_2)}{\pi_{s2}}{(Y,\psi)}
\end{tikzpicture}
\]
are {\em rationally $H$-equivalent} if there exists nonzero $q\in \Q$ such that
\begin{enumerate}
\item ($\pi_{s1}^s \otimes id_{\Q}) \circ (\pi_{u1}^{s*} \otimes id_{\Q}) = q \cdot ((\pi_{s2}^s \otimes id_{\Q}) \circ (\pi_{u2}^{s*} \otimes id_{\Q}$)) and
\item ($\pi_{u1}^u \otimes id_{\Q}) \circ (\pi_{s1}^{u*}\otimes id_{\Q}) = q \cdot((\pi_{u2}^u \otimes id_{\Q}) \circ (\pi_{s2}^{u*} \otimes id_{\Q}))$,
\end{enumerate}
as maps on the rationalization of Putnam's homology theory (e.g., $\pi_{s1}^s \otimes id_{\Q} : H^s(M_1,\mu_1)\otimes \Q \rightarrow H^s(Y,\psi)\otimes \Q$).
\end{definition}

\begin{prop} \label{relEquCor}
We have the following relationships between the four notions of the equivalence: 
\begin{enumerate}
\item Isomorphic implies rationally isomorphic, $H$-equivalent, and rational $H$-equivalent. 
\item $H$-equivalent implies rationally $H$-equivalent.
\item Rationally isomorphic implies rationally $H$-equivalent.
\end{enumerate}
\end{prop}
\begin{proof}
The proof of the first two items is routine. For the final item, suppose that
\[
\begin{tikzpicture}
\Corr{0}{0}{(X,\varphi)}{\pi_{u1}}{(M_1,\mu_1)}{\pi_{s1}}{(Y,\psi)}
\node at (4.35,1) { and };
\Corr{6}{0}{(X,\varphi)}{\pi_{u2}}{(M_2,\mu_2)}{\pi_{s2}}{(Y,\psi)}
\end{tikzpicture}
\]
are rationally isomorphic and $\theta_1: (M,\mu) \rightarrow (M_1, \mu_1)$ and $\theta_2: (M,\mu) \rightarrow (M_2, \mu_2)$ are both s- and u-bijective where $\theta_1$ is $n_1$-to-one and $\theta_2$ is $n_2$-to-one, as in Definition \ref{ratIso}. We show that there exists nonzero $q\in \mathbb{Q}$ such that
\begin{equation}\label{rat-isom to rat-H-equiv form}
(\pi_{s1}^s \otimes id_{\Q}) \circ (\pi_{u1}^{s*} \otimes id_{\Q}) = q \cdot \left( ( \pi_{s2}^s \otimes id_{\Q}) \circ (\pi_{u2}^{s*} \otimes id_{\Q}) \right),
\end{equation}
and note that the proof of Definition \ref{ratHEquiv} (2) is analogous.

To show \eqref{rat-isom to rat-H-equiv form}, Proposition \ref{bothSUmap} implies that 
\[
\theta_1^{s} \circ \theta_1^{s*}= n_1 \cdot id_{H^s(M_1,\mu_1)} \text{ and } \theta_2^{s} \circ \theta_2^{s*}= n_2 \cdot id_{H^s(M_2,\mu_2)}
\]
The result now follows from the commutative diagram in Definition \ref{ratIso} and the computation:
\begin{align*}
 n_1 \cdot (\pi^s_{s1} \circ \pi^{s*}_{u1}) & = (\pi_{s1} \circ \theta_1)^s \circ (\pi_{u1} \circ \theta_1)^{s*} \\
 & = (\pi_{s2}\circ \theta_2)^s \circ (\pi_{u2} \circ \theta_2)^{s*} \\
 & = n_2 \cdot (\pi^s_{s2} \circ \pi^{s*}_{u2}). \qedhere
\end{align*}
\end{proof}

We note that the reverse implications in Proposition \ref{relEquCor} are all false, see Example \ref{revImp}.

\begin{example}\label{LMconstruction}
Consider the set-up of Theorem \ref{corConFromMatrix} and the Lind-Marcus process for constructing a correspondence employed in the proof.  Let 
\[
\begin{tikzpicture}
\Corr{0}{0}{(\Sigma_G,\sigma)}{\pi_{u1}}{(\Sigma_{M_1},\sigma)}{\pi_{s1}}{(\Sigma_H,\sigma)}
\node at (4.35,1) { and };
\Corr{6}{0}{(\Sigma_G,\sigma)}{\pi_{u2}}{(\Sigma_{M_2},\sigma)}{\pi_{s2}}{(\Sigma_H,\sigma)}
\end{tikzpicture}
\]
be two correspondences constructed from the Lind-Marcus process by making different ``choices'' (see \cite[p.g. 287]{LM}). Theorem \ref{corConFromMatrix} implies that the correspondences are H-equivalent.  It is straightforward to come up with examples where different choices lead to  $\Sigma_{M_1}$ and $\Sigma_{M_2}$ which are not conjugate, and hence correspondences which are not isomorphic. In fact, such examples can be constructed with $\Sigma_G$ and $\Sigma_H$ both equal to the full 2-shift.
\end{example}

\begin{remark} \label{compositionAndEquivalenceBrady}
It is worth noting that each of the four notions of equivalence of correspondences ``behave well'' with respect to composition of correspondences.  More precisely, if 
\[
\begin{tikzpicture}
\Corr{0}{0}{(X,\varphi)}{\pi_{u}}{(M,\mu_1)}{\pi_{s}}{(Y,\psi)}
\end{tikzpicture}
\]
is a correspondence and 
\[
\begin{tikzpicture}
\Corr{0}{0}{(Y,\psi)}{\pi_{u1}}{(M_1,\mu_1)}{\pi_{s1}}{(Z,\xi)}
\node at (4.35,1) { and };
\Corr{6}{0}{(Y,\psi)}{\pi_{u2}}{(M_2,\mu_2)}{\pi_{s2}}{(Z,\xi)}
\end{tikzpicture}
\]
are correspondences which are equivalent under one of the four notions of equivalence. Then 
\begin{equation}\label{compositions}
\begin{tikzpicture}
\Corr{0}{0}{(X,\varphi)}{\pi_{u} \circ P_1}{\fib{(M,\mu)}{\pi_{s}}{\pi_{u1}}{(M_1,\mu_1)}}{\pi_{s1} \circ P_2}{(Z,\xi)}
\node at (5,1) {and };
\Corr{7}{0}{(X,\varphi)}{\pi_{u} \circ P_1}{\fib{(M,\mu)}{\pi_{s}}{\pi_{u2}}{(M_2,\mu_2)}}{\pi_{s2} \circ P_2}{(Z,\xi)}
\end{tikzpicture}
\end{equation} 
are equivalent (under the same notion of equivalence).  The equivalence relation for which this is the least obvious is \emph{rational isomorphism}, but in this case if
\[
\begin{tikzpicture}
\Corr{0}{0}{(M_1,\mu_1)}{\theta_1}{(M',\mu')}{\theta_2}{(M_2,\mu_2)}
\end{tikzpicture}
\]
is a correspondence which gives the rational isomorphism (see Definition \ref{ratIso}), then 
\[
\begin{tikzpicture}
\Corr{0}{0}{\hspace{-2cm}\fib{(M,\mu)}{\pi_{s}}{\pi_{u1}}{(M_1,\mu_1)}}{id \times \theta_1}{\fib{(M,\mu)}{\pi_{s}}{\pi_{u1}\circ \theta_1}{(M',\mu')}}{id \times \theta_2}{\hspace{2cm}\fib{(M,\mu)}{\pi_{s}}{\pi_{u2}}{(M_2,\mu_2)}}
\end{tikzpicture}
\]
gives a rational isomorphism between the correspondences in \eqref{compositions}. We leave it to the reader to verify the details.
\end{remark}

\subsection{Equivalence of Smale spaces}\label{equiv_Smale}

Each notion of equivalence leads to a natural notion of an invertible correspondence:

\begin{definition}[Strongly Invertible]\label{strong_inv}
A correspondence
\[
\begin{tikzpicture}
\Corr{-0.5}{0}{(X,\varphi)}{\pi_u}{(M,\mu)}{\pi_s}{(Y,\psi)}
\node at (6.5,1.15) {is {\em strongly invertible} };
\node at (6.5,0.65) {(respectively {\em rationally invertible}) if };
\node at (6.5,0.15) { there exists a correspondence};
\Corr{10.5}{0}{(Y,\psi)}{\pi_u'}{(M',\mu')}{\pi_s'}{(X,\varphi)}
\end{tikzpicture}
\]
such that the composition in one order is isomorphic (respectively, rationally isomorphic) to the identity correspondence on $(X,\varphi)$ and in the other order is isomorphic (respectively, rationally isomorphic) to the identity on $(Y,\psi)$. The second correspondence will be called the strong (respectively rational) inverse of the first.
\end{definition}

\begin{definition}[$H$-Invertible]\label{weak_inv}
A correspondence
\[
\begin{tikzpicture}
\Corr{0}{0}{(X,\varphi)}{\pi_u}{(M,\mu)}{\pi_s}{(Y,\psi)}
\node at (6.5,1) {is {\em $H$-invertible} if };
\node at (6.5,0.5) { there exists a correspondence};
\Corr{10}{0}{(Y,\psi)}{\pi_u'}{(M',\mu')}{\pi_s'}{(X,\varphi)}
\end{tikzpicture}
\]
such that
\begin{align*}
(\pi_s')^s \circ (\pi_u')^{s *} \circ \pi_s^s \circ \pi_u^{s *} &= \id_{H^s(X,\varphi)} \\
\pi_u^u \circ \pi_s^{u *} \circ (\pi_u')^u \circ (\pi_s')^{u *} &= \id_{H^u(X,\varphi)} \\
\pi_s^s \circ \pi_u^{s *} \circ (\pi_s')^s \circ (\pi_u')^{s *} &= \id_{H^s(Y,\psi)} \\
(\pi_u')^u \circ (\pi_s')^{u *} \circ \pi_u^u \circ \pi_s^{u *} &= \id_{H^u(Y,\psi)}.
\end{align*}
In other words, the composition of the correspondences in each of the possible orders is $H$-equivalent to the identity correspondence. In a similar way, a correspondence is {\em rationally $H$-invertible} if there exists a correspondence such that the  the composition of the correspondences in each of the possible orders is rationally $H$-equivalent to the identity correspondence. 
\end{definition}

\begin{remark}
Inverses are not unique, but are unique up to the appropriate notion of equivalence of correspondences.  For example, ``the" rational inverse to a given rationally invertible correspondence is unique up to rational isomorphism. This follows from a short argument which uses the definition of inverse and Remark \ref{compositionAndEquivalenceBrady}.  

\end{remark}
In summary, we have four notions of invertibility for a correspondence. Associated to each notion of invertibility is an equivalence relation on Smale spaces: 
\begin{definition} \label{equivSmaleSpaceLevel}
Given Smale spaces $(X,\varphi)$ and $(Y, \psi)$, we write $(X,\varphi) \sim (Y, \psi)$ if there exists an invertible correspondence from $(X,\varphi)$ to $(Y,\psi)$. We denote these four equivalence relations on Smale spaces respectively by $\sim_{strong}$, $\sim_{H}$, $\sim_{Rat}$, and $\sim_{Rat-H}$. 
\end{definition}
We note that equivalence for correspondences is distinct (but related) to equivalence for Smale spaces. Compare $H$-equivalence for correspondences in Definition \ref{HEquiv} with $H$-equivalence for Smale spaces, defined immediately above; context should make clear whether a given equivalence is between Smale spaces or correspondences.

\begin{prop} \label{relEquSmale}
The relations, $\sim_{strong}$, $\sim_{H}$, $\sim_{Rat}$, and $\sim_{Rat-H}$ are equivalence relations. Moreover, we have the following relationships between the four notions of the equivalence on Smale spaces: 
\begin{enumerate}
\item $\sim_{strong}$ implies $\sim_{H}$, $\sim_{Rat}$, and $\sim_{Rat-H}$. 
\item $\sim_{H}$ implies $\sim_{Rat-H}$.
\item $\sim_{Rat}$ implies $\sim_{Rat-H}$.
\end{enumerate}
\end{prop}
\begin{proof}
The proof that each relation is an equivalence relation is similar in each case; we give the details in the case of $\sim_{strong}$. The identity correspondence is a strongly invertible correspondence, hence for any Smale space $(X,\varphi)$, we have $(X,\varphi) \sim_{strong} (X, \varphi)$. If $(X,\varphi) \sim_{strong} (Y, \psi)$, then, by definition, there exists strongly invertible correspondence from $(X, \varphi)$ to $(Y, \psi)$. The strong inverse of this correspondence is a strongly invertible correspondence from $(Y,\psi)$ to $(X, \varphi)$. Finally, if $(X, \varphi) \sim_{strong}(Y, \psi)$ and $(Y, \psi) \sim_{strong} (Z, \zeta)$, then a fibre product construction similar to one in the proof of Proposition \ref{finiteEquIsAnEqu} implies that $(X, \varphi) \sim_{strong} (Y, \psi)$.
The second part of the proposition follows from Proposition \ref{relEquCor}.
\end{proof}
\begin{prop} \label{proRatHEquImpEquRatHom}
Let $(X,\varphi)$ and $(Y,\psi)$ be Smale spaces. If $(X,\varphi) \sim_{rat-H} (Y,\psi)$, then 
\begin{align}
\label{proRatHEquImpEquRatHom_1}
(H^s(X,\varphi)\otimes \Q, (\varphi^{-1})^s\otimes id) &\cong (H^s(Y,\psi)\otimes \Q, (\psi^{-1})^s\otimes id) \quad \text{ and} \\
\label{proRatHEquImpEquRatHom_2}
(H^u(X,\varphi)\otimes \Q, \varphi^u\otimes id) &\cong (H^s(Y,\psi)\otimes \Q, \psi^u\otimes id).
\end{align}
\end{prop}
\begin{proof}
We prove \eqref{proRatHEquImpEquRatHom_1} and note that \eqref{proRatHEquImpEquRatHom_2} is analogous. To prove \eqref{proRatHEquImpEquRatHom_1} we will construct an explicit isomorphism from a given rationally $H$-invertible correspondence from $(X,\varphi)$ to $(Y,\psi)$. As such, fix respectively such a correspondence and its inverse
\[
\begin{tikzpicture}
\Corr{0}{0}{(X,\varphi)}{\pi_{u1}}{(M_1,\mu_1)}{\pi_{s1}}{(Y,\psi)}
\node at (4.35,1) { and };
\Corr{6}{0}{(Y,\psi)}{\pi_{u2}}{(M_2,\mu_2)}{\pi_{s2}}{(X,\varphi).}
\end{tikzpicture}
\]
By the definition of rational $H$-invertibility (see also Definition \ref{ratHEquiv}) there exists a non-negative rational number $q$ such that
\begin{align*}
(\pi_{s2}^s \circ \pi^{s*}_{u2}\circ \pi_{s1}^s \circ \pi_{u1}^{s*})\otimes id_{\Q} & = q \cdot id_{H^s(X,\varphi)\otimes \Q} \quad \text{ and} \\
(\pi_{s1}^s \circ \pi^{s*}_{u1}\circ \pi_{s2}^s \circ \pi_{u2}^{s*})\otimes id_{\Q} & = q \cdot id_{H^s(Y,\psi)\otimes \Q}.
\end{align*}
Hence $\varphi:=(\pi_{s1}^s \circ \pi_{u1}^{s*}) \otimes id_{\Q}$ is an isomorphism, which interwines the maps on homology induced from the map defining the dynamics.
\end{proof}
\begin{thm} \label{StrInvIfAndOnlyIfConj}
There exists a strongly invertible correspondence
\[
\begin{tikzpicture}
\Corr{0}{0}{(X,\varphi)}{\pi_u}{(M,\mu)}{\pi_s}{(Y,\psi)}
\end{tikzpicture}
\]
if and only if $(X,\varphi)$ is conjugate to $(Y,\psi)$.
\end{thm}

\begin{proof}
To prove that a conjugacy leads to a strongly invertible correspondence, let $\Phi: (X, \varphi) \to (Y,\psi)$ be a conjugacy.  Then
\[
\begin{tikzpicture}
\Corr{0}{0}{(X,\varphi)}{id}{(X,\varphi)}{\Phi}{(Y,\psi)}
\node at (4.35,1) { and };
\Corr{6}{0}{(Y,\psi)}{id}{(Y,\psi)}{\Phi^{-1}}{(X,\varphi)}
\end{tikzpicture}
\]
are strong inverses. 

For the other direction, let 
\[
\begin{tikzpicture}
\Corr{0}{0}{(Y,\psi)}{\pi_u'}{(M',\mu')}{\pi_s'}{(X,\varphi)}
\end{tikzpicture}
\]
be a strong inverse to the correspondence in the statement of the theorem. Then
\[
\begin{tikzpicture}
\Corr{0}{0}{(X,\varphi)}{\pi_{u}\circ P_1}{\fib{(M,\mu)}{\pi_{s}}{\pi_{u}'}{(M',\mu')}}{\pi_{s}'\circ P_2}{(X,\varphi)}
\end{tikzpicture}
\]
is strongly equivalent to the identity correspondence with a conjugacy $\Theta: \fib{(M,\mu)}{\pi_{s}}{\pi_{u}'}{(M',\mu')} \to (X,\varphi) $.  It follows that $\pi_{u}\circ P_1 = \Theta = \pi_{s}'\circ P_2$.  Therefore $\pi_{u}$ and $\pi_{s}'$ are conjugacies.  A similar argument, composing in the opposite order, shows that $\pi_{u}'$ and $\pi_{s}$ are conjugacies, and therefore $(X,\varphi)$ is conjugate to $(Y,\psi)$.
\end{proof}

We are now able to justify the statement in the introduction that strong equivalence of Smale spaces generalizes strong shift equivalence for shifts of finite type.

\begin{cor}\label{SFTtostongshiftequiv}
Let $(\Sigma, \sigma)$ be a shift of finite type and $(X,\varphi)$ a Smale space. Then $(\Sigma, \sigma)$ is strongly equivalent to $(X,\varphi)$ if and only if $(X,\varphi)$ is a shift of finite type and $(\Sigma, \sigma)$ and $(X,\varphi)$ are strong shift equivalent.
\end{cor}

\begin{proof}
This result follows from the previous result and \cite[Theorem 7.2.7]{LM}. 
\end{proof}

\begin{lemma} \label{mustBeShiftLemma}
Let $(\Sigma, \sigma)$ be a shift of finite type and $(X,\varphi)$ a Smale space. Suppose that there exist correspondences
\[
\begin{tikzpicture}
\Corr{0}{0}{(X,\varphi)}{\pi_{u1}}{(M_1,\mu_1)}{\pi_{s1}}{(\Sigma,\sigma)}
\node at (4.35,1) { and };
\Corr{6}{0}{(\Sigma,\sigma)}{\pi_{u2}}{(M_2,\mu_2)}{\pi_{s2}}{(X,\varphi)}
\end{tikzpicture}
\]
Then $(X,\varphi)$ is a shift of finite type.
\end{lemma}
\begin{proof}
The fact that $\Sigma$ is totally disconnected and $\pi_{s1}$ is a finite-to-one factor map from $(M_1, \mu_1)$ to $(\Sigma, \sigma)$ imply that $M_1$ is also totally disconnected. Similarly, $M_2$ is also totally disconnected. Then \cite[Theorem 2.2.8]{put} implies that both $(M_1, \mu_1)$ and $(M_2, \mu_2)$ are shifts of finite type. Furthermore, the fact that $\pi_{u1}$ is u-bijective and $\pi_{s2}$ is s-bijective imply that $(X,\varphi)$ has totally disconnected unstable and stable sets; hence it is also a shift of finite type.
\end{proof}

\begin{cor}\label{equivtoSFTimpliesSFT}
Let $(\Sigma, \sigma)$ be a shift of finite type and $(X,\varphi)$ a Smale space. If $(\Sigma, \sigma)$ and $(X, \varphi)$ are equivalent using any of the four notions of equivalence considered in the statement of Proposition \ref{relEquSmale}, then $(X,\varphi)$ is a shift of finite type.
\end{cor}

\begin{proof}
The result follows immediately from Lemma \ref{mustBeShiftLemma}
\end{proof}

The following theorem justifies our claim that $H$-equivalence of Smale spaces generalizes shift equivalence for shifts of finite type.

\begin{thm} \label{HequiImpliesShiftSFT}
Let $(\Sigma_A, \sigma_A)$ be a shift of finite type and $(X,\varphi)$ a Smale space. Then $(\Sigma_A, \sigma_A)$ is $H$-equivalent to $(X,\varphi)$ if and only if $(X,\varphi)$ is a shift of finite type and $(\Sigma_A, \sigma_A)$ and $(X,\varphi)$ are shift equivalent.
\end{thm}

\begin{proof}
Corollary \ref{equivtoSFTimpliesSFT} implies that if $(\Sigma_A, \sigma_A)$ is $H$-equivalent to $(X,\varphi)$, then $(X,\varphi)$ is a shift of finite type. Thus, it suffices to prove the the theorem in the case that $(X,\varphi)=(\Sigma_B, \sigma)$ is a shift of finite type.

For the forward direction, suppose 
\[
\begin{tikzpicture}
\Corr{0}{0}{(\Sigma_A, \sigma)}{\pi_u}{(M,\mu)}{\pi_s}{(\Sigma_B, \sigma)}
\end{tikzpicture}
\]
is an $H$-invertible correspondence. It follows that $(M, \mu)$ is also a shift of finite type. Moreover, the dimension groups of the two shifts of type are isomorphic; that is, $D^s(\Sigma_A, \sigma) \cong D^s(\Sigma_B, \sigma)$ and $D^u(\Sigma_A, \sigma) \cong D^u(\Sigma_B, \sigma)$. By Proposition \ref{ResOrdAndAutSFT}, this isomorphism preserves the order and automorphism induced from the shift. Then \cite[Theorem 7.5.8]{LM} implies that $(\Sigma_A, \sigma)$ and $(\Sigma_B, \sigma)$ are shift equivalent.

For the reverse direction, suppose $(\Sigma_A, \sigma)$ and $(\Sigma_B, \sigma)$ are shift equivalent (see \cite[Definition 7.3.1]{LM}). By definition, there exist non-negative integer matrices $R$ and $S$ such that $AR = RB$, $SA = BS$, $RS = A^l$, and $SR = B^l$ for some $l\in \N$.
Theorem \ref{corConFromMatrix} implies that there are correspondences
\begin{equation}\label{pic:corresp in 5.15}
\begin{tikzpicture}
\Corr{0}{0}{(\Sigma_A, \sigma)}{\pi_{u}}{(\Sigma, \sigma)}{\pi_{s}}{(\Sigma_B, \sigma)}
\node at (4.35,1) { and };
\Corr{6}{0}{(\Sigma_B, \sigma)}{\pi_{u}'}{(\Sigma', \sigma)}{\pi_{s}'}{(\Sigma_A, \sigma)}
\end{tikzpicture}
\end{equation}
such that 
\[
(\pi_s^s \circ \pi_u^{s*})[v,j] = [Sv,j] \quad \text{ and } \quad (\pi_s')^s \circ (\pi_u')^{s*}[w,j] = [Rw,j].
\]
Now consider the correspondence
\[
\begin{tikzpicture}
\Corr{0}{0}{(\Sigma_A,\sigma)}{id}{(\Sigma_A,\sigma)}{\sigma^l}{(\Sigma_A,\sigma)}
\Corr{3}{0}{}{\pi_{u}}{(\Sigma,\sigma)}{\pi_{s}}{(\Sigma_B,\sigma)}
\Corr{1.5}{1.3}{}{P_1}{(\Sigma_A,\sigma) _{\sigma^l}\hspace{-0.1cm}\times_{\pi_{u}} \hspace{-0.1cm}(\Sigma,\sigma)}{P_2}{}
\end{tikzpicture}
\]
This is an $H$-inverse to the correspondence on the right hand side of \eqref{pic:corresp in 5.15} since
\begin{align*}
(\pi_s')^s \circ (\pi_u')^{s*} \circ \pi_s^s \circ \pi_u^{s*} \circ (\sigma^l)^s \circ id^{s*} [v,j] &= (\pi_s')^s \circ (\pi_u')^{s*} \circ \pi_s^s \circ \pi_u^{s*} [v, j+l] \\
&= (\pi_s')^s \circ (\pi_u')^{s*} [Sv,j+l] \\
&= [RSv, j+l] = [A^lv, j+l] \\
&= [v,j]
\end{align*}
and
\begin{align*}
\pi_s^s \circ \pi_u^{s*} \circ (\sigma^l)^s \circ id^{s*} \circ (\pi_s')^s \circ (\pi_u')^{s*} [w,j] &= \pi_s^s \circ \pi_u^{s*} \circ (\sigma^l)^s \circ id^{s*} [Rw,j] \\
&= \pi_s^s \circ \pi_u^{s*} [Rw,j+l] \\
&= [SRw, j+l] = [B^lw,j+l] \\
&= [w,j+l].
\end{align*}
The maps on $H^u(\: \cdot \:)$ behave similarly, and hence the correspondence on the right hand side of \eqref{pic:corresp in 5.15} is a $H$-invertible correspondence between $(\Sigma_A,\sigma)$ and $(\Sigma_B,\sigma)$.
\end{proof}

\begin{example} \label{revImp}
Based on Theorem \ref{HequiImpliesShiftSFT}, the William's conjecture \cite{Wil} that shift equivalence classifies shifts of finite type,  can be reformulated in the language of correspondences as follows: does $H$-equivalence imply strong equivalence? Counterexamples to the William's conjecture are given in \cite{KRred} and in the irreducible case in \cite{KRirr}. Similarly, one can ask whether Rat-$H$-equivalence implies $H$-equivalence? However, this turns out to be false as well; explicit examples are not difficult to construct.
\end{example}

\subsection{Existence equivalences}\label{equiv_exist}

In the previous section, we saw that each notion of invertibility for correspondences leads to a notion of equivalence for Smale spaces. However, following the general idea of finite equivalence and almost conjugacy, one can define an equivalence relation based on the existence of correspondences.

\begin{definition}\label{def:cor-equiv}
We say that two Smale spaces $(X,\varphi)$ and $(Y,\psi)$ are cor-equivalent if there exist correspondences:
\[
\begin{tikzpicture}
\Corr{0}{0}{(X,\varphi)}{\pi_u}{(M,\mu)}{\pi_s}{(Y,\psi)}
\node at (4.35,1) { and };
\Corr{6}{0}{(Y,\psi)}{\pi_u'}{(M',\mu')}{\pi_s'}{(X,\varphi)}
\end{tikzpicture}
\]
If this condition holds, we write $(X, \varphi) \sim_{cor} (Y,\psi)$.
\end{definition}

The reader should note that we make no assumption on the maps on homology induced from the correspondences in the definition of cor-equivalence and that cor-equivalence implies finite equivalence by \cite[Theorem 2.5.3]{put}.

%

The proof of the next result is similar to that of Theorem \ref{HequiImpliesShiftSFT} and is omitted; we note that both Theorem \ref{838LMinCor} and Lemma \ref{mustBeShiftLemma} are relevant.

\begin{prop}
The relation $\sim_{cor}$ is an equivalence relation. Moreover, if $(\Sigma, \sigma)$ is a shift of finite type and $(X,\varphi)$ is a Smale space, then $(\Sigma, \sigma) \sim_{cor} (X,\varphi)$ if and only if $(X,\varphi)$ is a shift of finite type and $(\Sigma, \sigma)$ and $(X, \varphi)$ are finitely equivalent.
\end{prop}

\begin{example}
The $2^{\infty}$-solenoid, see for example \cite[Section 7.3]{put}, and the full two shift are finitely equivalent, but not cor-equivalent.
\end{example}

We also have a relation which is related to almost conjugacy. 

\begin{definition}
We say that two irreducible Smale spaces, $(X,\varphi)$ and $(Y,\psi)$, are Almost Cor Conjugate (i.e., ACC-equivalent) if there exist correspondences:
\[
\begin{tikzpicture}
\Corr{0}{0}{(X,\varphi)}{\pi_u}{(M,\mu)}{\pi_s}{(Y,\psi)}
\node at (4.35,1) { and };
\Corr{6}{0}{(Y,\psi)}{\pi_u'}{(M',\mu')}{\pi_s'}{(X,\varphi)}
\end{tikzpicture}
\]
where each of the maps, $\pi_u$, $\pi_s$, $\pi_u'$, and $\pi_s'$, are almost one-to-one. If this condition holds, we write $(X, \varphi) \sim_{ACC} (Y,\psi)$.
\end{definition}

Again, the proof of the next result is similar to that of Theorem \ref{HequiImpliesShiftSFT} and is omitted; in this case \cite[Theorem 9.3.2]{LM} is relevant.

\begin{prop}
The relation $\sim_{ACC}$ is an equivalence relation. Moreover, if $(\Sigma, \sigma)$ is an irreducible shift of finite type and $(X,\varphi)$ is an irreducible Smale space, then $(\Sigma, \sigma) \sim_{ACC} (X,\varphi)$ if and only if $(X,\varphi)$ is a shift of finite type and $(\Sigma, \sigma)$ and $(X, \varphi)$ are almost conjugate.
\end{prop}

\begin{example}
The $2^{\infty}$-solenoid and the full two shift are almost conjugate, but not ACC-equivalent.
\end{example}

\subsection{Correspondences and dimension}\label{sec:dimension}
Lemma \ref{mustBeShiftLemma} points to a fundamental, but at this point vague, difference between a finite equivalence and a correspondence: the former does not ``see" dimension, while the latter seems to. More precisely, Bowen's theorem implies that if two Smale spaces are finitely equivalent, then the middle space can be taken to be a shift of finite type; this cannot be done in general for a correspondence. 

Our main goal in this section is to study the implications of a cor-equivalence on the covering dimensions of the various Smale spaces appearing in the equivalence. We make use of several basic facts about covering dimension and direct the reader to \cite{HurWal} for the appropriate background material.  We begin with a general observation about finite-to-one maps. The proof of the next result follows from a theorem of Hurewicz, see for example \cite[Theorem VI 7]{HurWal}.
\begin{lemma}
Let $\pi: (X,\varphi) \rightarrow (Y, \psi)$ be a finite-to-one factor map between Smale spaces. Then, $\dim(X) \le \dim(Y)$. 
\end{lemma}
\begin{cor} \label{finEquAndDim}
If 
\[
\begin{tikzpicture}
\Corr{0}{0}{(X,\varphi)}{\pi_X}{(M,\mu)}{\pi_Y}{(Y,\psi)}
\end{tikzpicture}
\]
is a finite equivalence of Smale spaces, then $\dim(M) \le \min \{ \dim(X), \dim(Y)\}$.
\end{cor}
We move to the case of cor-equivalence. As such let
\[
\begin{tikzpicture}
\Corr{0}{0}{(X,\varphi)}{\pi_u}{(M,\mu)}{\pi_s}{(Y,\psi)}
\node at (4.35,1) { and };
\Corr{6}{0}{(Y,\psi)}{\pi_u'}{(M',\mu')}{\pi_s'}{(X,\varphi)}
\end{tikzpicture}
\]
be two correspondences (i.e., an explicit cor-equivalence between $(X,\varphi)$ and $(Y, \psi)$).

To simplify the discussion, we assume that $\dim (X^s(x))$ and $\dim (X^u(x))$ are independent of $x\in X$ and likewise $\dim (Y^s(y))$ and $\dim (Y^u(y))$ are independent of $y\in Y$ (as we will see in Remark \ref{dimIndep}, this assumption holds quite broadly). In particular, this assumption implies that for each $x\in X$ and $y\in Y$, 
\begin{equation}\label{consDimXY}
\dim(X) = \dim (X^s(x)) + \dim (X^u(x)) \hbox{ and } \dim(Y)= \dim(Y^s(y)) + \dim(Y^u(y)).
\end{equation}
In addition, the following inequality comes from basic properties of dimension and the definitions of $s$- and $u$-bijective maps: for any $m\in M$,
\[
\dim(M) \ge \dim(M^s(m)) + \dim(M^u(m)) = \dim( Y^s(\pi_s(m))) + \dim( X^u(\pi_u(m))).
\]
Likewise, for any $m'\in M'$,
\[
\dim(M') \ge \dim((M')^s(m')) + \dim((M')^u(m'))=  \dim( Y^u(\pi'_u(m'))) + \dim( X^s(\pi'_u(m'))).
\]
Using these two inequalities and Equation \eqref{consDimXY}, we obtain
\[ \dim(M) + \dim(M') \ge \dim(X)+ \dim(Y) \ge \min \{ 2\dim(X), 2\dim(Y)\}.
\]
However, Corollary \ref{finEquAndDim} implies that the reverse inequality holds as well. It follows that
\begin{align*}
\dim(M) + \dim(M') &=  \min \{ 2\dim(X), 2\dim(Y)\} \, \hbox{ and } \\
\dim(X)&=\dim(Y)=\dim(M)=\dim(M'),
\end{align*}
where the latter equality uses $\dim(M) + \dim(M') \ge \dim(X)+ \dim(Y)$ and Corollary \ref{finEquAndDim}. 

Furthermore, since there exists $m\in M$ such that $\dim(M)= \dim(M^s(m)) + \dim(M^u(m))$, and using properties of $s$- and $u$-bijective maps, we have $x \in X$ and $y\in Y$ such that $\dim(M) = \dim(X^u(x)) + \dim(Y^s(y))$. Recalling that 
$\dim(X) = \dim(X^u(x)) + \dim(X^s(x))$, for any $x \in X$, and that $\dim(Y^s(y))$ is independent of $y\in Y$, we have that $\dim(X^s(x))=\dim(Y^s(y))$ for each $x\in X$ and $y\in Y$. Likewise, $\dim(X^u(x))=\dim(Y^u(y))$ for each $x\in X$ and $y\in Y$. We summarize this discussion in a theorem:
\begin{thm}\label{thm:cor to dim}
Let $(X,\varphi)$ and $(Y,\psi)$ be Smale spaces and let
\[
\begin{tikzpicture}
\Corr{0}{0}{(X,\varphi)}{\pi_u}{(M,\mu)}{\pi_s}{(Y,\psi)}
\node at (4.35,1) { and };
\Corr{6}{0}{(Y,\psi)}{\pi_u'}{(M',\mu')}{\pi_s'}{(X,\varphi)}
\end{tikzpicture}
\]
be correspondences. Furthermore, assume that $\dim (X^s(x))$ and $\dim (X^u(x))$ are independent of $x\in X$ and likewise $\dim (Y^s(y))$ and $\dim (Y^u(y))$ are independent of $y\in Y$. Then
\begin{enumerate}
\item $\dim(X)=\dim(Y)=\dim(M)=\dim(M')$ and
\item For each $x\in X$ and $y\in Y$, $\dim(X^s(x))=\dim(Y^s(y))$ and $\dim(X^u(x))=\dim(Y^u(y))$.
\end{enumerate}
\end{thm}

\begin{remark}\label{dimIndep}
The assumption of the independence of the dimension of the stable/unstable sets holds quite broadly.  In particular, for a shift of finite type all of the stable and unstable sets have dimension zero.  Moreover, the following proposition shows that the assumption holds for any irreducible Smale space.  In general, however, the assumption does not hold for a non-wandering Smale space. For example, consider the disjoint union of a shift of finite type and a hyperbolic toral automorphism.
\end{remark}

\begin{prop}
Let $(X,\varphi)$ be an irreducible Smale space.  Then $dim(X^s(x))$ (respectively $dim(X^u(x))$) is independent of $x \in X$.
\end{prop}

\begin{proof}
We prove the $X^s(x)$ case, the $X^u(x)$ case is similar. Since $(X,\varphi)$ is irreducible, we may choose $x \in X$ such that $\{\varphi^n(x) \}_{n=0}^{\infty}$ is dense.  Note that $dim(X^s(\varphi^i(x))) = dim(X^s(x))$ for all $i$, and call this common dimension $D$.  Now let $x' \in X^s(x)$ be such that $dim(X^s(x',\epsilon_X)) = dim(X^s(x)) = D$.  Since $x' \sim_s x$ we have that $\{\varphi^n(x')\}_{n=0}^{\infty}$ is also dense in $X$, and for each $n$, $\varphi^n(x')$ has neighbourhood $\varphi^n(X^s(x',\epsilon_X)) \subset X^s(\varphi^n(x')) = X^s(\varphi^n(x))$ with dimension D. Let $U_n = \varphi^n(X^s(x',\epsilon_X))$.  Now, let $z \in X$ be any point.  We can find $n$ such that $d(z,\varphi^n(x')) < \epsilon_X/2$ and the diameter of $U_n < \epsilon_X/2$. Then $[\varphi^n(x'), X^s(z, \epsilon_X/2)]$ is an open set in $X^s(\varphi^n(x'))$ which is homeomorphic to $X^s(z, \epsilon_X/2)$, hence $dim(X^s(z, \epsilon_X/2)) \leq D$. On the other hand, $[z,U_n]$ is an open set in $X^s(z)$ homeomorphic to $U_n$, hence $dim(X^s(z)) \geq dim([z,U_n]) = D$. It follows that $dim(X^s(z)) = D$. 
\end{proof}

\section{Correspondences and equivalence for general Smale spaces}\label{sec:constructioncorr}

In this section we provide methods for constructing correspondences. In particular, we prove a special case of a K\"{u}nneth formula for Putnam's homology in which one of the Smale spaces in a shift of finite type and the other has totally disconnected stable sets. It is used to construct examples of Smale spaces with isomorphic homology theories that are not $H$-equivalent.

\subsection{General constructions}

We now give a partial generalization of \cite[Theorem 8.3.8]{LM} (see also Theorem \ref{838LMinCor}).

\begin{thm}\label{generalized LM 8.3.8}
Suppose $(X,\varphi)$ and $(Y,\psi)$ are irreducible Smale spaces such that $X^u(x)$ is totally disconnected for each $x \in X$ and $Y^s(y)$ is totally disconnected for each $y \in Y$. Then, the following are equivalent:
\begin{enumerate}
\item\label{6.1_item1} $(X,\varphi)$ and $(Y,\psi)$ are finitely equivalent;
\item\label{6.1_item2} $(X,\varphi)$ and $(Y,\psi)$ have equal entropy; and
\item\label{6.1_item3} there is correspondence from $(X,\varphi)$ to $(Y, \psi)$.
\end{enumerate}
\end{thm}

\begin{proof}
Theorem \ref{finiteEqualEntropy} gives the implication \eqref{6.1_item1} $\implies$ \eqref{6.1_item2}. The relevant definitions and \cite[Theorem 2.5.3]{put} give the implication \eqref{6.1_item3} $\implies$ \eqref{6.1_item1}.   

We now show that \eqref{6.1_item2} $\implies$ \eqref{6.1_item3}. Given $(X,\varphi)$ and $(Y,\psi)$,  Bowen's Theorem \cite[Theorem 28]{Bow} implies that there exist shifts of finite $(\Sigma_X, \sigma)$ and $(\Sigma_Y, \sigma)$ and factor maps $\rho_X : (\Sigma_X, \sigma) \rightarrow (X,\varphi)$ and $\rho_Y : (\Sigma_Y, \sigma) \rightarrow (Y, \psi)$. Moreover, the assumptions on the unstable and stable sets of $(X,\varphi)$ and $(Y, \psi)$ imply that we can take $\rho_X$ to be u-bijective and $\rho_Y$ to be s-bijective.

The existence of the maps $\rho_X$ and $\rho_Y$ and the fact that $(X,\varphi)$ and $(Y,\psi)$ are irreducible with equal entropy imply that $(\Sigma_X, \sigma)$ and $(\Sigma_Y, \sigma)$ can be taken to be irreducible and have equal entropy. Theorem \ref{838LMinCor} implies that there exists a correspondence:
\[
\begin{tikzpicture}
\Corr{0}{0}{(\Sigma_X,\sigma)}{\pi_u}{(\Sigma_M,\sigma)}{\pi_s}{(\Sigma_Y,\sigma)}
\end{tikzpicture}
\]
where $(\Sigma_M, \sigma)$ is a shift of finite type. It follows that
\[
\begin{tikzpicture}
\Corr{0}{0}{(X,\varphi)}{\rho_X \circ \pi_u}{(\Sigma_M,\sigma)}{\rho_Y \circ \pi_s}{(Y,\psi)}
\end{tikzpicture}
\]
is a correspondence.
\end{proof}
\begin{example}
Let $(X_0, \varphi_0)$, $(X_1, \varphi_1)$, and $(X_2, \varphi_2)$ be nonwandering Smale spaces and $\rho_s : (X_1, \varphi_1) \rightarrow (X_0, \varphi_0)$ be $s$-bijective and $\rho_u: (X_2, \varphi_2) \rightarrow (X_0, \varphi_0)$ be $u$-bijective. Then 
\[
\begin{tikzpicture}
\Corr{0}{0}{(X_1,\varphi_1)}{P_1}{\fib{(X_1,\varphi_1)}{\rho_s}{\rho_u}{(X_2,\varphi_2)}}{P_2}{(X_2,\varphi_2)}
\end{tikzpicture}
\]
is a correspondence. Moreover, \cite[Theorem 5.1]{DKWfunProPutHom} and Remark \ref{nonWanCon} imply that the induced maps on homology are equal:
$$ \rho_u^{s*} \circ \rho_u^s=P_2^s \circ P_1^{s*} \hbox{ and } \rho_s^{u*} \circ \rho_u^u = P_1^u \circ P_2^{u*}.$$
A specific example of this situation is an s/u-bijective pair, see Definition \ref{s/u-bijective pair}. More generally, one can construct correspondences from $(X_1, \varphi_1)$ to $(X_2, \varphi_2)$ with middle space given by the ``iterated fiber product", see \cite[Section 2.6]{put}.
\end{example}

Another way to produce correspondences between Smale spaces is to use Putnam's Lifting Theorem, which we restate for the readers convenience. 

\begin{thm}[Putnam's Lifting Theorem {\cite[Theorem 1.1]{PutLift}}] \label{putLiftThm}
Let $(X,\varphi)$ and $(Y,\psi)$ be irreducible Smale spaces and $\pi: (X,\varphi) \rightarrow (Y,\psi)$ be an almost one-to-one factor map. Then, there exist irreducible Smale spaces $(\tilde{X}, \tilde{\varphi})$ and $(\tilde{Y}, \tilde{\psi})$ along with a commutative diagram
\begin{center}
$\begin{CD}
(\tilde{X},\tilde{\varphi}) @>\tilde{\pi}>> (\tilde{Y}, \tilde{\psi}) \\
@V\alpha VV @VV \beta V \\
(X, \varphi) @>\pi >> (Y, \psi)
\end{CD}$
\end{center}
where
\begin{enumerate}
\item $\tilde{\pi}$ is s-resolving;
\item $\alpha$ and $\beta$ are u-resolving.
\end{enumerate}
\end{thm} 

Using the notation of Theorem \ref{putLiftThm} and \cite[Theorem 2.5.8]{put}, we have that 
\[
\begin{tikzpicture}
\Corr{0}{0}{(X,\varphi)}{\alpha}{(\tilde{X},\tilde{\varphi})}{\tilde{\pi}}{(\tilde{Y},\tilde{\psi})}
\end{tikzpicture}
\]
is a correspondence. Given the setup of Theorem \ref{putLiftThm}, it is natural to ask if there exists a correspondence 
\[
\begin{tikzpicture}
\Corr{0}{0}{(X,\varphi)}{\pi_u}{(M,\mu)}{\pi_s}{(Y,\psi)}
\end{tikzpicture}
\]
which ``factors" $\pi$ (i.e., $\pi \circ \pi_u =\pi_s$). This would correspond to being able to take $\beta$ equal to the identity map in the diagram in the statement of Putnam's lifting theorem, which is not possible in general. We note that factoring $\pi$ as the composition of an s-bijective map with a u-bijective map corresponds to being able to take $\alpha$ equal to the identity in this same diagram; the reader is directed to the introduction of \cite{PutLift} for more on this case.

Conversely, it is natural to ask if, for each correspondence
\[
\begin{tikzpicture}
\Corr{0}{0}{(X,\varphi)}{\pi_u}{(M,\mu)}{\pi_s}{(Y,\psi),}
\end{tikzpicture}
\]
 there exists a factor map $\pi: (X, \varphi) \to (Y,\psi)$ such that $\pi \circ \pi_u=\pi_s$. However, there are shifts of finite type that are counterexamples. In general, if there exists such a factor map $\pi$, then for each $x \in X$, $\pi_s ( \pi_u^{-1} (\{ x\}))$ is a single element; this rather strong condition is certainly not satisfied by an arbitrary correspondence.

\subsection{Special case of a K\"{u}nneth formula for Putnam's homology}

In this section, we prove the following result:
\begin{thm} \label{babyKunn}
Let $(X,\varphi)$ be a non-wandering Smale space with totally disconnected stable sets, and $(\Sigma, \sigma)$ be a non-wandering shift of finite type. Then,
\begin{align}
\label{Kunneth_1}
\left( H^s_*((\Sigma, \sigma)\times (X,\varphi)), (\sigma \times \varphi)^s \right) & \cong
\left(D^s(\Sigma, \sigma) \otimes H^s_*(X, \varphi), \sigma^s \otimes \varphi^s \right) \, \text{ and}\\
\label{Kunneth_2}
\left( H^u_*((\Sigma, \sigma)\times (X,\varphi)), (\sigma \times \varphi)^u \right) & \cong
\left(D^u(\Sigma, \sigma) \otimes H^u_*(X, \varphi), \sigma^u \otimes \varphi^u \right)
\end{align}
\end{thm}
The proof requires a number of lemmas. We give a detailed proof of \eqref{Kunneth_1} and note that \eqref{Kunneth_2} is analogous. The reader should note also that the assumption that the stable sets are totally disconnected simplifies the definition of Putnam's homology theory, see \cite[Chapter 4 and Section 7.2]{put}.

\begin{lemma} \label{KunForDimGroup}
If $(\Sigma, \sigma)$ and $(\Sigma^{\prime}, \sigma^{\prime})$ are shifts of finite type, then 
\begin{align*}
(D^s((\Sigma, \sigma) \times (\Sigma^{\prime}, \sigma^{\prime})), (\sigma \times \sigma)^s) & \cong (D^s(\Sigma, \sigma) \otimes D^s(\Sigma^{\prime}, \sigma^{\prime}), \sigma^s \otimes (\sigma')^s) \\
(D^u((\Sigma, \sigma) \times (\Sigma^{\prime}, \sigma^{\prime})), (\sigma \times \sigma)^u) & \cong (D^u(\Sigma, \sigma) \otimes D^u(\Sigma^{\prime}, \sigma^{\prime}), \sigma^u \otimes (\sigma')^u).
\end{align*}
\end{lemma}
\begin{proof}
We prove the stable case; the unstable case is similar.  We can find graphs $G$ and $H$ such that  $(\Sigma, \sigma) \cong (\Sigma_G, \sigma_G)$ and $(\Sigma^{\prime}, \sigma^{\prime}) \cong (\Sigma_H, \sigma_H)$. Then
\[
D^s(\Sigma_G, \sigma_G):= \Z^{|G^0|} \stackrel{A_G}{\rightarrow} \Z^{|G^0|} \stackrel{A_G}{\rightarrow} \Z^{|G^0|}
 \stackrel{A_G}{\rightarrow} 
 \]
 and 
 \[
D^s(\Sigma_H, \sigma_H):= \Z^{|H^0|} \stackrel{A_H}{\rightarrow} \Z^{|H^0|} \stackrel{A_H}{\rightarrow} \Z^{|H^0|}
 \stackrel{A_H}{\rightarrow},
 \]
as described in Section \ref{DimensionGroups}. Now consider the graph $G \times H$.  That is, the graph with vertex set $(G \times H)^0 = G^0 \times H^0$ and if $G^1_{uv}$ is the set of edges in $G$ from $u$ to $v$, then $(G \times H)^1_{(u,u')(v,v')} = G^1_{uv} \times H^1_{u'v'}$.  Now, it is straightforward to see that $(\Sigma_G, \sigma_G) \times (\Sigma_H, \sigma_H) \cong (\Sigma_{G \times H}, \sigma_{G \times H})$ and that $A_{G \times H} = A_G \otimes A_H$. So, identifying $\Z^{|G^0 \times H^0|}$ with $\Z^{|G^0|} \otimes \Z^{|H^0|}$, we have 
\[
D^s(\Sigma_{G \times H}, \sigma_{G \times H}):= \Z^{|G^0|} \otimes \Z^{|H^0|} \stackrel{A_G \otimes A_H}{\longrightarrow} \Z^{|G^0|} \otimes \Z^{|H^0|} \stackrel{A_G \otimes A_H}{\longrightarrow} \Z^{|G^0|} \otimes \Z^{|H^0|}
 \stackrel{A_G \otimes A_H}{\longrightarrow} 
 \]  
It is then routine to show that $D^s(\Sigma_{G \times H}, \sigma_{G \times H}) \cong D^s(\Sigma_G, \sigma_G) \otimes D^s(\Sigma_H, \sigma_H)$.
\end{proof}

\begin{lemma} \label{proOfSelfFibPro}
Suppose $(X,\varphi)$, $(\tilde{X}, \tilde{\varphi})$, $(Y, \psi)$, and $(\tilde{Y}, \tilde{\psi})$ are Smale spaces and $\rho_1: (X,\varphi) \rightarrow (\tilde{X}, \tilde{\varphi})$ and $\rho_2: (Y, \psi) \rightarrow (\tilde{Y}, \tilde{\psi})$ are factor maps. For any positive integer $N$,
\[
X_N(\rho_1) \times Y_N(\rho_2) \cong (X\times Y)_N(\rho_1 \times \rho_2).
\]
In particular, if $\rho_1$ is the identity map, we have
\[
X \times Y_N(\rho_2) \cong (X\times Y)_N(id \times \rho_2).
\]
\end{lemma}
\begin{proof}
The first statement follows from the definitions; the second is an immediate consequence of the first.
\end{proof}
\begin{lemma}\label{dimGrpIso} 
Let $(\Sigma, \sigma)$ and $(\Sigma_X,\sigma)$ be shifts of finite type, and let $\rho: (\Sigma_X,\sigma) \to (X,\varphi)$ be an s-bijective factor map.  Fix a non-negative integer $N$, and let 
\[
\delta^{id \times \rho}_i : D^s ((\Sigma \times \Sigma_X)_N(id \times \rho)) \rightarrow D^s((\Sigma \times \Sigma_X)_{N-1}(id \times \rho))
\]
and 
\[
\delta^{\rho}_i : D^s ((\Sigma_X)_N(\rho)) \rightarrow D^s((\Sigma_X)_{N-1}( \rho))
\]
be defined as in \cite[Section 4.1]{put}. Then, using the isomorphisms in the previous two lemmas, we have that the following diagram commutes:
\begin{center}
$\begin{CD}
D^s ((\Sigma \times \Sigma_X)_N(id \times \rho)) @>\delta^{id\times \rho}_i >>  D^s((\Sigma \times \Sigma_X)_{N-1}(id \times \rho)) \\
@V\Phi VV @VV\Phi V \\
 D^s(\Sigma,\sigma) \otimes D^s ((\Sigma_X)_N(\rho)) @>id \otimes \delta^{\rho}_i >> D^s(\Sigma,\sigma) \otimes D^s((\Sigma_X)_{N-1}( \rho))
\end{CD}$
\end{center}
where $\Phi$ is the isomorphism defined in the proof.
\end{lemma}
\begin{proof}
By Lemma \ref{proOfSelfFibPro}, we have the following commutative diagram:
\begin{center}
$\begin{CD}
(\Sigma \times \Sigma_X)_N(id \times \rho) @>\delta^{id\times \rho}_i >>  (\Sigma \times \Sigma_X)_{N-1}(id \times \rho) \\
@VVV @VVV \\
\Sigma \times (\Sigma_X)_N(\rho) @>id \times \delta^{\rho}_i >> \Sigma \times (\Sigma_X)_{N-1}( \rho)
\end{CD}$
\end{center}
where the vertical maps are conjugacies. In addition, by Lemma \ref{KunForDimGroup}, we also have the commutative diagram:
\begin{center}
$\begin{CD}
D^s(\Sigma \times (\Sigma_X)_N(\rho)) @>(id \times \delta^{\rho}_i)^s >> D^s(\Sigma \times (\Sigma_X)_{N-1}( \rho)) \\
@VVV @VVV \\
D^s(\Sigma) \otimes D^s((\Sigma_X)_N(\rho)) @>id^s \otimes (\delta^{\rho}_i)^s >> D^s(\Sigma) \otimes D^s((\Sigma_X)_{N-1}(\rho))
\end{CD}$
\end{center}
where again the vertical maps are isomorphisms. These commutative diagrams and \cite[Theorem 3.4.1]{put} yield the required result; the map $\Phi$ is explicitly given by the composition of the isomorphisms on stable homology induced from the vertical isomorphisms in the two diagrams.
\end{proof}

\begin{cor}\label{babyKunnethCor}
The isomorphism $\Phi$ defined in the proof of Lemma \ref{dimGrpIso} gives
\[
\left( D^s((\Sigma \times \Sigma_X)_N(id \times \rho)), d^s_N(id\times \rho) \right) \cong
\left( D^s(\Sigma) \otimes D^s((\Sigma_X)_N(\rho)) , id\otimes d^s_N(\rho) \right) 
\]
\end{cor}

\begin{proof}[Proof of Theorem \ref{babyKunn}]
Since $(X,\varphi)$ has totally disconnected stable sets, there exists a non-wandering shift of finite type $(\Sigma_X,\sigma)$ and s-bijective factor map $\rho: (\Sigma_X,\sigma) \to (X, \varphi)$.  The result then follows from Corollary \ref{babyKunnethCor} since the isomorphism $\Phi$ respects the functors associated with s-bijective and u-bijective maps.
\end{proof}

\subsection{Isomorphic homology vs H-equivalence}
The next example shows that the notion of $H$-equivalence for Smale spaces is strictly stronger than the property of having an isomorphism at the level of the homology that intertwines the maps induced by the dynamics.
\begin{example}
Let $(Y, \psi)$ denote the solenoid obtained from the pre-solenoid, $(X,h)$, in \cite[Example 2.7]{APG}. Based on \cite[Example 3.4 / 3.5]{APG}, we have that
\begin{align*}
H^s_N(Y, \psi) & \cong \left\{ \begin{array}{ccl} \{ (i, i+j) \ | \ i \in \mathbb{Z} \left[ \frac{1}{3} \right], \: j\in \mathbb{Z} \} / 2 \mathbb{Z} (-1, 1)  & : & N=0 \\ 0 & : & N \neq 0 \end{array} \right. \\
H^u_N(Y, \psi) & \cong \left\{ \begin{array}{ccl} \mathbb{Z} \left[ \frac{1}{3} \right] & : & N=0 \\ \mathbb{Z}/2\mathbb{Z} & : & N=1 \\ 0 & : & N \neq 0, 1 \end{array} \right.
\end{align*}
and that the actions on the degree zero groups of $(\psi^{-1})^s$ and $\psi^u$ are respectively
\begin{enumerate}
\item multiplication by the matrix $ \left( \begin{array}{cc} 2 & 1 \\ 1 & 2 \end{array} \right)$;
\item multiplication by three.
\end{enumerate}
Let $(\Sigma_{[2]}, \sigma_{[2]})$ denote the full two shift; its homology groups are its dimension groups, which are both $\mathbb{Z}\left[\frac{1}{2}\right]$. The action of $\sigma_{[2]}^s$ is division by two, and $\sigma_{[2]}^u$ is multiplication by two. Then, Theorem \ref{babyKunn} and a computation imply that
\begin{align*}
H^s_N((\Sigma_{[2]}, \sigma_{[2]}) \times (Y, \psi)) & \cong \left\{ \begin{array}{ccl} \mathbb{Z}\left[ \frac{1}{6} \right]  & : & N=0 \\ 0 & : & N \neq 0 \end{array} \right. \\
H^u_N(\Sigma_{[2]}, \sigma_{[2]}) \times (Y, \psi)) & \cong \left\{  \begin{array}{ccl} \mathbb{Z}\left[ \frac{1}{6} \right]  & : & N=0 \\ 0 & : & N \neq 0. \end{array} \right.
\end{align*}
These homology groups are the same as those of the full six shift, $(\Sigma_{[6]}, \sigma_{[6]})$. Moreover, the maps $(\sigma_{[2]} \times \psi)^s$ and $\sigma_{[6]}^s$ are both division by six, while $(\sigma_{[2]} \times \psi)^u$ and $\sigma_{[6]}^u$ are both multiplication by six.


Thus, $(\Sigma_{[2]}, \sigma_{[2]}) \times (Y, \psi)$ and $(\Sigma_{[6]}, \sigma_{[6]})$ have isomorphic homology theories (even with the automorphisms associated to the dynamics accounted for). We note however that, since $(\Sigma_{[2]}, \sigma_{[2]}) \times (Y, \psi)$ is not a shift of finite type, Lemma \ref{mustBeShiftLemma} implies that it is not $H$-equivalent to $(\Sigma_{[6]}, \sigma_{[6]})$. Thus, $H$-equivalence is strictly stronger than the property of having an automorphism preserving isomorphism at the level of homology.
\end{example}

In \cite{WiePhD,WiePaper}, Wieler gives a very general construction, using inverse limits, of Smale spaces with totally disconnected stable sets. She also presents an intriguing example based on the Sierpinski gasket. Putnam has told us that his computations with T. Bazett show that this Smale space has the same homology as the full 3-shift, although the details have not appeared. Thus, this ``gasket Smale space" is an example of a Smale space which is not a shift of finite type, but has the same homology as one. The example discussed in this section also has this property. Other examples with this property can be produced by considering the product of an unorientable solenoid with the full two shift, and then following the construction in the previous example.

\section{Implications of equivalences for Smale spaces}\label{sec:implications}

In this section, we present a number of consequences, in terms of dynamic notions, of correspondence equivalences between two Smale spaces. The starting point for this discussion is the diagram on page 261 of \cite{LM}; it outlines the various consequences of (strong) shift equivalence for shifts of finite type. One goal of this section is to illustrate the similarity, through their dynamic consequences, between the notions of equivalence for Smale spaces defined using correspondences and (strong) shift equivalence. Quite a number of the implications below follow almost immediately from Putnam's work on the $\zeta$-function of a Smale space in \cite[Section 6]{put}.

We should emphasize that the idea that Smale spaces and shifts of finite type (should) share many similar properties is certainty not new, see for example any of \cite{AM, APG, Bow, Fri, Par, put, Rue1, Sma, WiePaper}. Indeed, Bowen's theorem \cite{Bow} and the existence of an s/u-bijective pair \cite[Theorem 2.6.3]{put} (also see Definition \ref{s/u-bijective pair}) provide strong links between the theories. In the context of Putnam's homology theory we have the following table, which lists generalizations of notions from shift of finite type theory to the theory of Smale spaces. The final four entries will be defined and justified in the remaining part of this section. That these notions are generalizations follows for the most part from results in \cite{put}. 

\begin{center}
\begin{tabular}{|c|c|}
\hline
Shift of Finite Type & Smale Space \\
\hline
$(\Sigma_G, \sigma)$ & $(X, \varphi)$ \\
$D^{s}(\Sigma_G, \sigma)$, $D^u(\Sigma_G, \sigma)$ & $H^s_{\ast}(X,\varphi)$, $H^u_{\ast}(X,\varphi)$ \\
$A_{G\ast}$, $A_G^{\ast}$ & $(\varphi^{-1})^s$, $\varphi^u$ \\
$\mathcal{R}_{A_G}$ & $H^s_{\ast}(X,\varphi) \otimes \mathbb{Q}$ \\
$A_G^{\times}$ & $(\varphi^{-1})^s \otimes id_{\mathbb{Q}}$ \\
$J^{\times}(A_G)$ & $J\left((\varphi^{-1})^s \otimes id_{\mathbb{Q}}\right)$ \\
$\zeta_{\sigma_A}(t)$ & $\zeta_{\varphi}(t)$ \\
\hline
\end{tabular}
\end{center}

\begin{definition}[{\cite[Definitions 7.4.2, 7.4.5, and 7.4.9]{LM}}]
Suppose A is an $r \times r$ integral matrix. The \emph{eventual range} of $A$ is the subspace of $\mathbb{Q}^r$ given by
\[
\mathcal{R}_A = \bigcap_{k=1}^{\infty}A^k\mathbb{Q}^r.
\]
The \emph{invertible part} $A^\times$ of $A$ is the linear transformation obtained by restricting $A$ to its eventual range; that is, $A^\times:\mathcal{R}_A \to \mathcal{R}_A$ is defined by $A^\times=(\sigma_A^{-1})^s \otimes_{\Z} id$. Finally, the \emph{Jordan form away from zero} $J^\times(A)$ is the Jordan form of $A$ restricted to its eventual range.
\end{definition}

We note that our definition of the eventual range of a matrix corresponds to the eventual range of $A^t$ in \cite[Definition 7.4.2]{LM}. We choose to do things differently here so that we can state our results for $D^s(\Sigma_G,\sigma)$ rather than $D^u(\Sigma_G,\sigma)$.

\begin{prop}
Let $(\Sigma_G, \sigma)$ be a shift of finite type. Then 
\[
\mathcal{R}_{A_G} \cong D^s(\Sigma_G,\sigma)\otimes \Q
\]
and we have the following commutative diagram:
\[
\begin{CD}
\mathcal{R}_{A_G} @>A_G>> \mathcal{R}_{A_G} \\
@VVV @VVV \\
D^s(\Sigma_G, \sigma)\otimes \Q @>(\sigma^{-1})^s\otimes id_{\Q} >> D^s(\Sigma_G, \sigma)\otimes \Q
\end{CD}
\]
\end{prop} 
\begin{proof}
That $\mathcal{R}_{A_G} \cong D^s(\Sigma_G,\sigma)\otimes \Q$ follows from the definitions.
The commutativity of the diagram follows from the fact that, for $[v,n] \otimes q \in D^s(\Sigma_G, \sigma)\otimes \Q$, $((\sigma^{-1})^s\otimes id_{\Q})([v,n] \otimes q) = [A_G v,n] \otimes q$.
\end{proof}

Let $(X,\varphi)$ be a Smale space. Then the $\zeta$-function of $(X,\varphi)$ is defined by
\[
\zeta_{\varphi}(t)=\exp \left( \sum_{n=1}^{\infty} \frac{{\rm card}\{ x \in X | \varphi^n(x)=x\}}{n} t^n \right)
\]
The $\zeta$-function of a Smale space is discussed in detail in \cite[Section 6]{put}. In particular, we will use the following result.

\begin{cor}[{\cite[Corollary 6.1.2]{put}}]\label{cor-put6-1-2}
Let $(X,\varphi)$ be an irreducible Smale space. For each $N\in \Z$ and $t\in \R$, let $p_N(t)$ be the determinant of 
\[ Id-t(\varphi^{-1})^s_N \otimes id_{\R} : H^s_N(X,\varphi)\otimes \R \rightarrow H^s_N(X,\varphi)\otimes \R
\]
Then
\[ \zeta_{\varphi}(t)= \prod_{N\in \Z} p_N(t)^{(-1)^{N+1}}.
\]
We note that for all but finitely many $N$ we have that $p_N(t)=1$.
\end{cor}

\begin{prop}\label{h-equiv to zeta}
If $(H^s(X,\varphi)\otimes \R, (\varphi^{-1})^s)$ and $(H^s(Y,\psi)\otimes \R, (\psi^{-1})^s)$ are isomorphic, then $\zeta_{\varphi}(s)=\zeta_{\psi}(s)$.
\end{prop}

Note that if $(X,\varphi) \sim_{Rat-H} (Y,\psi)$, then the hypothesis of Proposition \ref{h-equiv to zeta} is automatic.

\begin{proof}
Let $\Phi$ be the isomorphism which which intertwines the maps induced by $\varphi^{-1}$ and $\psi^{-1}$. For example, in the case $(X,\varphi) \sim_{Rat-H} (Y,\psi)$ one can take the map constructed in Proposition \ref{proRatHEquImpEquRatHom}. One then checks that, for each $N\in \Z$ and $t\in \R$, the following diagram commutes:
\[
\begin{CD}
H^s_N(X,\varphi)\otimes \R @> Id - t(\varphi^{-1})_N^s \otimes id_\R >> H^s_N(X,\varphi)\otimes \R \\
@V \Phi VV @VV \Phi V \\
H^s_N(Y,\psi)\otimes \R @> Id - t(\psi^{-1})_N^s \otimes id_\R >> H^s_N(Y,\psi)\otimes \R \\
\end{CD}
\]
Corollary \ref{cor-put6-1-2} then implies the result.
\end{proof}

We now have the following analogue of the diagram on \cite[p.261]{LM} for Smale spaces:

\begin{center}
\begin{tikzpicture}
\node at (0,0) {$(X,\varphi) \sim_{strong} (Y,\psi)$};
\draw[implies-implies,double equal sign distance] (-2,0) -- node[above=0.2cm,pos=0.5]{\small{Thm \ref{StrInvIfAndOnlyIfConj}}} (-3.5,0);
\node at (-5,0) {$(X,\varphi) \cong (Y,\psi)$};
\node at (-7,-4) {$(H^s(X,\varphi),(\varphi^{-1})^s)\simeq(H^s(Y,\psi),(\psi^{-1})^s)$};
\node at (-7,-4.5) {and};
\node at (-7,-5.1) {$(H^u(X,\varphi),\varphi^u)\simeq(H^u(Y,\psi),\psi^u)$};
\draw[->,>=implies,double equal sign distance] (-3.7,-3) -- node[above=0.2cm,pos=0.5]{\small{Defn \ref{weak_inv}}} (-6.2,-3.5);
\draw[->,>=implies,double equal sign distance] (-0.5,-0.5) -- node[left=0.2cm,pos=0.4]{\small{Prop \ref{relEquSmale}}} (-1.5,-2.5);
\node at (-2,-3) {$(X,\varphi) \sim_{H} (Y,\psi)$};
\draw[->,>=implies,double equal sign distance] (0.5,-0.5) -- node[right=0.2cm,pos=0.4]{\small{Prop \ref{relEquSmale}}} (1.5,-2.5);
\node at (2,-3) {$(X,\varphi) \sim_{Rat} (Y,\psi)$};
\draw[->,>=implies,double equal sign distance] (-1.5,-3.5) -- node[left=0.2cm,pos=0.4]{\small{Prop \ref{relEquSmale}}} (-0.5,-5.5);
\draw[->,>=implies,double equal sign distance] (1.5,-3.5) -- node[right=0.2cm,pos=0.4]{\small{Prop \ref{relEquSmale}}} (0.5,-5.5);
\node at (0,-6) {$(X,\varphi) \sim_{Rat-H} (Y,\psi)$};
\draw[->,>=implies,double equal sign distance] (-7,-5.5) --  (-7,-6.5);
\draw[->,>=implies,double equal sign distance] (-2.2,-6) -- node[above=0.2cm,pos=0.5]{\small{Prop \ref{proRatHEquImpEquRatHom}}} (-4.7,-6.5);
\node at (-7,-7) {$(H^s(X,\varphi)\otimes \Q,(\varphi^{-1})^s \otimes id)\simeq(H^s(Y,\psi)\otimes \Q,(\psi^{-1})^s \otimes id)$};
\node at (-7,-7.5) {and};
\node at (-7,-8.1) {$(H^u(X,\varphi)\otimes \Q,\varphi^u \otimes id)\simeq(H^u(Y,\psi)\otimes \Q,\psi^u \otimes id)$};
\draw[->,>=implies,double equal sign distance] (-7,-8.6) -- node[left=0.2cm,pos=0.5]{\small{Prop \ref{h-equiv to zeta}}} (-7,-9.5);
\node at (-7,-10) {$\zeta_{\varphi}(t) =\zeta_{\psi}(t)$};
\draw[->,>=implies,double equal sign distance] (-7,-10.5) -- node[above=0.2cm,pos=0.5]{} (-7,-11.5);
\node at (-7,-12) {$per(X,\varphi)=per(Y,\psi)$};
\draw[->,>=implies,double equal sign distance] (0.5,-6.5) -- node[right=0.2cm,pos=0.4]{\small{Def'n \ref{def:cor-equiv}}} (1.0,-7.5);
\node at (1,-8) {$(X,\varphi) \sim_{cor} (Y,\psi)$};
\draw[->,>=implies,double equal sign distance] (1,-8.5) -- node[right=0.2cm,pos=0.4]{\small{Def'n \ref{def:finite-equiv}}} (1.0,-9.5);
\node at (1,-10) {$(X,\varphi) \sim_{fin} (Y,\psi)$};
\draw[->,>=implies,double equal sign distance] (1,-10.5) -- node[right=0.2cm,pos=0.4]{\small{Prop \ref{finiteEqualEntropy}}} (1.0,-11.5);
\node at (1,-12) {$h(X,\varphi) = h(Y,\psi)$};
\end{tikzpicture}
\end{center}

\section{Outlook: Further functors}\label{sec:outlook}

It is likely that a correspondence will also induce maps between the K-theory groups of stable and unstable C*-algebras associated to Smale spaces and even, in a certain sense, at the level of the $C^*$-algebras themselves. However, there are a few issues to be addressed to do so: 
\begin{enumerate}
\item Currently, a C*-algebra can only be associated to a Smale space when it is non-wandering.
\item The functoriality of the C*-algebras under s/u-bijective maps is subtle, see \cite{Put3} (and again only known in the non-wandering setting).
\item There is no ``pullback lemma"  in the case of general Smale spaces or K-theory. Other examples of ``pullback lemmas" are \cite[Theorem 3.5.11]{put} in the case of the dimension group of shifts of finite type and \cite[Theorem 5.1]{DKWfunProPutHom} in the case of Putnam's homology theory.
\end{enumerate}
While we believe that these issues can be overcome, at least in the case of K-theory, a detailed development would greatly increase the length of the current paper. As such, we will pursue this topic in a future paper. We note that the work of Thomsen in \cite{Tho} is likely to factor into solutions to these issues.

\end{document}